\documentclass[11 pt]{article}
\usepackage{multicol}
\usepackage{amsmath}
\usepackage{graphicx}
\usepackage{amssymb}
\usepackage{amsthm}
\usepackage{amsbsy}
\usepackage{indentfirst}
\usepackage[margin=1in]{geometry}
\usepackage{pstricks}
\usepackage{faktor}
\usepackage{float}
\usepackage{datetime}

\usepackage{yfonts}
\usepackage{amsfonts}
\usepackage{mathrsfs}
\usepackage{hyperref}
\usepackage{datetime}
\usepackage[ampersand]{easylist}
\usepackage{caption}
\usepackage{upgreek}
\usepackage{theoremref}
\usepackage{tikz-cd}
\newcommand\jy[1]{{ y}^{(#1)}}

 \definecolor{ForestGreen}{RGB}{34, 139,  34}

 \theoremstyle{definition}
 \newtheorem{result}{Result}
 
\newtheorem{theorem}{Theorem}[section]

\newtheorem{proposition}[theorem]{Proposition}
\newtheorem{lemma}[theorem]{Lemma}
\newtheorem{corol}[theorem]{Corollary}
\theoremstyle{definition}

\newtheorem{remark}[theorem]{Remark}
\newtheorem{example}[theorem]{Example}
\newtheorem{definition}[theorem]{Definition}

 \newcommand{\beq}{\begin{equation}}
\newcommand{\eeq}{\end{equation}}
\newcommand{\C}{\mathbb{C}}

\newcommand{\R}{\mathbb{R}}
\newcommand{\Z}{\mathbb{Z}}
\newcommand{\Q}{\mathbb{Q}}

\newcommand{\PP}{\mathbb{P}}

\newcommand{\Chi}{X}
\newcommand{\Chib}{\boldsymbol{X}}

\newcommand{\Af}{\mathcal{A}}
\newcommand{\SAf}{\mathcal{SA}}
\newcommand{\SE}{\mathcal{SE}}
\newcommand{\PGL}{\mathcal{PGL}}
\DeclareMathOperator\Sym{Sym}
\DeclareMathOperator\Stab{Stab}
\DeclareMathOperator\val{val}
\DeclareMathOperator\mult{mult}

\newcommand{\SymG}{\Sym(\Chi,G)}
\newcommand{\StabG}{\Stab(\Chi,G)}

\newcommand{\I}{\mathcal{I}}
\newcommand{\imS}{\mathcal{S}}
\newcommand{\s}{\sigma}

\newcommand{\bs}{\boldsymbol{\sigma}}
\newcommand{\bphi}{\boldsymbol{\phi}}

\newcommand{\F}{\textbf{F}}
\newcommand{\G}{\textbf{G}}
\newcommand{\T}{\textbf{T}}
\newcommand{\p}{\mathbf{p}}

\renewcommand{\L}{{\bf L}}

\newcommand{\ISE}{\I^{\mathcal{SE}}}
\newcommand{\IA}{\I^{\mathcal{A}}}
\newcommand{\IP}{\I^{\mathcal{P}}}

\newcommand{\ISA}{\I^{\mathcal{S}\mathcal{A}}}

\newcommand{\pp}{\mathbf{p}}
\newcommand{\ba}{\mathbf{a}}
\newcommand{\bc}{\mathbf{c}}
\newcommand{\Y}{\mathcal Y}
\newcommand{\pd}[2]{\frac{\partial #1}{\partial #2}}
\newcommand{\ord}{{\rm ord}}
\newtheorem{assum}[theorem]{Assumption}

\newcommand{\gdeg}{\tau}
\newcommand{\CP}{\C\PP^2}

\def\ssig{{\mathfrak S}}

\def\La{{\bf L_{\ba}}}
\def\Lat{{\bf L_{\tilde{\ba}}}}

\title{Differential signatures of algebraic curves}
\author{Irina A. Kogan, Michael Ruddy, Cynthia Vinzant \\ \small{North Carolina State University, Raleigh, NC, USA 27695}}
\usdate
\begin{document}
 \maketitle

%\tableofcontents

\abstract{In this paper, we adapt the differential signature construction to the equivalence problem for complex plane algebraic curves  under the actions of the projective group and its subgroups.  Given an action of a group $G$, a signature map assigns to a plane algebraic curve another  plane algebraic  curve (a signature curve) in such a way that two  generic curves have the same signatures  if and only if they  are $G$-equivalent.  We prove  that  for any $G$-action,   there exists  a pair of rational differential invariants, called classifying invariants, that  can be used to construct signatures. We derive a   formula for the degree of a signature curve in terms of the degree of the original  curve, the size of its symmetry group and some quantities  depending on a  choice  of classifying invariants.  We  show that  all generic curves have signatures of the same degree and this degree is the sharp upper bound. For the  full projective  group,  as well as for its affine, special affine and special Euclidean subgroups, we give explicit sets of rational classifying invariants and derive a formula  for the degree of the signature curve of a generic curve as a quadratic function of the degree of the original curve.}



% 14H50 plane and space curves
%14Qxx		Computational aspects in algebraic geometry, 14Q05  	Curves 
	%14L24  	Geometric invariant theory 
%68W30  	Symbolic computation and algebraic computation
%%	53A55  	Differential invariants (local theory), geometric objects%%
%68T45. 	Machine vision and scene understanding

\section{Introduction}
In the most general terms the group equivalence problem can be stated as follows: given  an action of a group $G$  on a set of objects, decide whether or not one  object can be transformed to another by a group element. An  elementary  geometry problem of deciding whether or not two triangles are congruent under  the  action of the group of rigid motions  (the Euclidean group) is an example. Many problems in  mathematics and applications can be reformulated in this manner,  and  equivalence  problems  are closely related to  many important  classification problems. 

The differential signature construction  originated from  Cartan's method for solving  equivalence problems for smooth manifolds  under  Lie group actions \cite{C53}. Signatures and, in particular, signatures of smooth curves gained popularity in many applications, such as image processing, computer vision,  and automated puzzle assembly  \cite{Bou00, CO98, HO14,grim-shakiban17}.   The differential signature construction for curves consists of the following steps: (1) an action of a group on a plane is prolonged to the jet space of  curves of sufficiently high order;  (2) on this jet space,  a pair of  independent differential invariants is constructed;  (3)  the restriction of this pair  to a given a curve parametrizes the signature curve.    Since the signature  is based  on invariants,  two equivalent curves have the same signature. The challenge lies in finding a pair of invariants so that (most) non-equivalent curves have different signatures.  In principle, such a pair of invariants can be found either by  the classical moving frame method formulated by Cartan \cite{C35} or by its modern generalization by Fels and Olver  \cite{FO99}, although in practice this may be challenging for large groups. 
The invariants obtained by the moving frame method are, in general, only locally defined and are designed to solve  local equivalence problems, i.e.~a problem of deciding whether or not  there exist  segments of two  smooth curves that are $G$-equivalent. The challenges  arising when using  these signatures for  solving  global equivalence problems for smooth curves, even in the case of the well-studied Euclidean action, underscored  are in   works  \cite{Hi12}, \cite{HO13} and \cite{MN09}.  

In  contrast with the smooth case,  any two irreducible algebraic curves that are locally equivalent are also globally  equivalent.
In addition, in the algebraic setting we can take advantage of  well-developed computational algebra algorithms to compute, compare and analyze signature curves. In order to take the full advantage of this machinery, we need to  build signature from \emph{rational invariants}, in which case the signature of an algebraic curve is again algebraic. The differential invariants obtained by the classical Cartan moving  frame method (called normalized invariants) or their counterparts obtained  Fels-Olver generalization (called replacement invariants) are not rational in general. In fact, an algebraic adaptation of the Fels-Olver given in \cite{hk:focm} shows that  local replacement  invariants, in general, are  algebraic over the field of global rational invariants. 

As  the first main result of this paper, we prove  existence of  two \emph{rational} differential invariants  that   can be used to construct signatures with good separation properties:  

\begin{result} Let $G\subset \PGL(3)$ be any closed algebraic subgroup of the projective group of positive dimension. Then there exists \emph{a pair of  rational differential invariants}, called \emph{classifying invariants}, of differential order at most equal to the $\dim G$, such that the signatures based on these invariants characterize equivalence classes of  generic  algebraic curves of degree $d$ for all  $d$ such that  $ \binom{d+2}{2} -2\geq \dim G$. See Theorem~\ref{Thm:ClassMain}.
\end{result}

Here and throughout the paper we formulate several results for a  {\it  generic curve} of degree $d$. This means that  there exists  a nonempty Zariski-open subset $\mathcal P_d$ of the  vector space $\C[x,y]_{\leq d}$ of all polynomials of degree at most $d$, such that a result is valid   for all  curves whose defining polynomials lie in $\mathcal P_d$.  
 
Given a pair of rational classifying invariants, the signature of a curve  $\Chi\subset \C^2$ is constructed as follows.  The restriction of classifying invariants to an algebraic curve $\Chi\subset \C^2$ defines a rational map  $\s_{\Chi}:\Chi\dashrightarrow \C^2$ called the \emph{signature map}.
Its image $\imS_{\Chi} = \s_{\Chi}(\Chi)$ is called the \emph{signature} of $\Chi$, and it is a Zariski-dense\footnote{The density statement is not valid over $\R$.  See, for instance, \cite[Example 1]{BKH}.} subset of its closure $\overline {\imS_{\Chi}}$, called the \emph{signature curve of $\Chi$}.   The defining polynomial for  the signature curve can be explicitly computed using elimination algorithms, as was studied in \cite{BKH}. However this computation is not always practically feasible and it is natural to ask what properties of  signature curves  can be determined  a priori. As the first step in this direction, we obtain a formula for the degree of the signature curve.
\begin{result}\label{res-degree} 
For a fixed algebraic group and a fixed set of classifying invariants, we derive  a formula for the degree of the 
signature curve of an algebraic curve in terms of the degree of the original curve, the size of its symmetry group, 
and  some quantities that depend on a  choice  of classifying invariants. See Theorem~\ref{Thm:Main1B}. We show that  signatures of generic curves all have the same degree and this degree provides the strict upper abound. See Theorem~\ref{Thm:Upperb}
\end{result}

One consequence of Theorem~\ref{th-class} is that, over $\C$, a classifying set of differential invariants can be computed by an algorithm for computing generators for the field of rational invariants, such as algorithms presented in \cite{DK15} and \cite{hk:jsc}. However the running time for these algorithms  can be prohibitively large for large groups, and these algorithms may produce a redundant set of generators from which  two appropriate invariants must be chosen. For the actions of the full projective group $\PGL(3)$ and its classical subgroups it appears more practical to build  rational classifying  invariants from the classical (non-rational) differential invariants.  
 We give explicit formulas for the classifying pairs for the special Euclidean $\SE(2)$, the special affine $\SAf(2)$, the affine $\Af(2)$ and projective $\PGL(3)$ groups. These groups  are especially relevant in computer vision and image processing.  We derive formulas for the degrees of signatures of generic curves  based on these pairs of invariants and show that these degrees are sharp upper bounds.
\begin{result}
For the actions of the full projective group $\PGL(3)$, and  its  subgroups such as the special Euclidean $\SE(2)$, the special affine $\SAf(2)$, the affine $\Af(2)$ and the classifying pairs of invariants  given by  \eqref{eq:InvariantsJet}, we find an upper bound for the degree of the algebraic signature of a plane curve of degree $d$. See Theorem~\ref{Thm:DegBound}.  From Result~\ref{res-degree}  we know this bound is tight for generic curves.
\end{result}

While the results are proved for complex curves under the action of  complex algebraic groups,  for many practical applications solving equivalence problems over the real field is important. For this reason, throughout the paper we often compare and contrast with the real  case. In particular, we would like to note that the pair of classifying invariants  \eqref{eq:InvariantsJet} can be proved to be classifying over $\R$ (see \cite{BKH}).  Therefore,   the signature of the real part of a complex curve $\Chi$ is contained in the   real part of the signature  curve $\overline {\imS_{\Chi}}$. Thus the degree results obtained in Theorem~\ref{Thm:DegBound} are also  applicable in the real case.

The paper is structured as follows. In Section~\ref{Sec:Differential} we review known results  about actions and invariants of algebraic groups, as well as the results about the jet spaces and differential invariants. We then prove our first main result about the existence  of a pair of classifying invariants.  Additionally we establish some basic facts about the relationship of the symmetry group of a curve and the curve's signature map, which play an important role in the degree formulas. In Section~\ref{Sec:Algebraic} we review some necessary definitions and theorems of algebraic geometry and prove our second  main result, which is a formula for the degree of the signature polynomial. In Section~\ref{Sec:Projective} we examine the signature polynomial for some specific examples of subgroups of the projective group and prove our third main result  about the degree of signatures of the generic curves for these groups. We also consider the family of Fermat curves, defined by polynomials $F_d(x,y)=x^d+y^d+1$, to show that the degree of a signature curve may be significantly lower than the generic degree. For this family, we give explicit formulas for signatures polynomials for all $d$ under the actions of the projective and affine group.   

Although the paper contains only few examples and computational details, the Maple code  and a large selection of examples are available on an online supplementary material page  \url{https://mgruddy.wixsite.com/home/dsag-supplementarymaterials}.

\paragraph{Acknowledgements.} 
We would like to thank Bojko Bakalov, Peter Olver, Kristian Ranestad, and Dmitry Zenkov  for helpful discussions and suggestions regarding this project. 
This work was supported in part  by the National Science Foundation grants DMS-1620014 and 
CCF-1319632.

\section{Differential invariants and signatures of algebraic curves} \label{Sec:Differential}
In this section, we prove our main  structural results about  the field of rational differential invariants and signatures of algebraic curves.
We start by reviewing,  in Section~\ref{Sec:AlgGroup},  known results  about actions and invariants of algebraic groups. In Section~\ref{Sec:Curves}, we consider the  action of the projective group and its subgroups on algebraic curves, give definitions of equivalence and symmetry for algebraic curves, and prove some useful results about the symmetry groups of curves (Propositions~\ref{Thm:SymEff}-\ref{pr:FOrbit}). In Section~\ref{Sec:Inv} we prolong the action to the jet space of curves and  define the notion of rational  classifying differential invariants. We  prove  an important structural result about the field of rational differential invariants (Theorem~\ref{th-gen}), as well  the existence of a classifying set (Theorem~\ref{th-class}). In Section~\ref{Sec:DiffCurve}, we show how differential invariants are evaluated on  an algebraic curve. We  define the notion of exceptional curves and show that  generic curves are non-exceptional (Theorem~\ref{thm-non-excep-generic}). Section~\ref{Sec:Sig}, we define the signature map and the signature curve of a non-exceptional  algebraic curve. We show that   signatures characterize the equivalence classes of generic algebraic curves (Theorem~\ref{Thm:ClassMain}) and prove that the signature map  of a curve $X$ is generically $n$ to one where $n$ is the cardinality of  the symmetry group of $X$ (Theorem~\ref{Thm:NToOne}).

\subsection{Actions and invariants of algebraic groups}\label{Sec:AlgGroup}
In this section, we review common  definitions and known results about actions and invariants of algebraic groups on algebraic varieties. The exposition follows \cite{vinberg89}, and we refer to this publication  for details, proofs, \ and further references.

 Throughout the section the ground field is $\C$ and the terms ``open'' and ``closed'' refer to Zariski topology.
An  {\it  algebraic group} is an algebraic variety equipped with a group structure. 
\begin{definition}Let $\Y$ be an affine or a projective variety. A {\it  rational action} 
of an algebraic group $G$ on $\Y$ is a rational map $\Phi\colon G \times \Y
  \dashrightarrow \Y$
that satisfies the following two properties:
\begin{enumerate}
\item $\Phi(e,p)=p$, $ \forall p\in \Y$, where $e$ is the identity in $G$, and 
 \item $\Phi(h, \Phi(g, p))=\Phi(h g, p)\),  for all $h,g\in G$ and $p\in \Y$, such that both sides are defined.
 \end{enumerate}
If the domain of $\Phi$ is all of $G \times \Y$ then $\Phi$ is a morphism and the action is called {\it  regular}.
\end{definition}

From now on, when the word ``action'' is used without an adjective, a rational action is assumed. We use the standard abbreviation $\Phi(g,p)=g\cdot p$ and state the following known definitions and results  used in our paper. %%%
\begin{definition}
For an action of $G$ on a variety $\Y$ and a point $p\in \Y$,  the  \textit{stabilizer of  $p$}  is the set
$$G_p=\{g\in G\ |\ g\cdot p=p\},$$
while the \emph{orbit of $p$} is the set
$$Gp=\{q\in \Y\ |\ \exists g\in G, g\cdot p=q\},$$
\end{definition}

We recall some basic properties of algebraic group actions. 

\begin{proposition}  \label{action-basics} Let $G$ be an algebraic group acting on an affine (or projective) variety $\Y$.
For any $p\in \Y$, the stabilizer $G_p$ is a closed algebraic subgroup of $G$.
The orbit $G p$ is a quasi-affine (or quasi-projective) variety  and 
 $$\dim G p=\dim G-\dim G_p.$$
 If $\Y$ is irreducible then the set of  all points whose orbit dimension  is less than maximal (equivalently the dimension of the stabilizer group is greater than minimum)  lies in a closed, proper subset of $\Y$.
 Finally, if $G$ is irreducible, then for all $p\in \Y$ the closure of the orbit $\overline{G p}$
 is irreducible.\end{proposition}

\begin{definition}A rational function $K$ on $\Y$ is  $G$-invariant if
$$ K(g\cdot p)=K(p), \text{ whenever both sides are defined}.$$
\end{definition}
%%%
The set of all rational  $G$-invariant functions is denoted by $\C(\Y)^G$. It is easy to see that it is a subfield of  the field $\C(\Y)$ of all rational functions on $\Y$.
%%%
\begin{definition}A subset $\I\subset \C(\Y)^G$ is called {\it  separating}  if there exists a nonempty open subset $W\subset \Y$ such that  for all $p,q\in W$, 
$$q\in Gp \ \Leftrightarrow \ K(p)=K(q)\text{ for all } K\in \I.$$ 
The  set $W$ is called a  {\it  domain of separation} for $\I$.
\end{definition}
Due to the Noetherian property, there exists a {\it  maximal} (with respect to inclusions) {\it  domain of separation}. It is not difficult to see that a maximal domain of separation  is a union of orbits, and therefore is a $G$-invariant set.
%%%%

In the following proposition, we summarize several important and non-trivial results about the structure of  $ \C(\Y)^G$.
See \cite{algeom4} or \cite{vinberg89} for details. 
%%%
\newpage
\begin{proposition}\hfill\label{inv-field-prop}
\begin{enumerate} 
\item The field  $\C(\Y)^G$ is finitely generated over $ \C$.
\item A subset $\I\subset \C(\Y)^G$ is generating if and only if it is separating.\footnote{In \cite{vinberg89}, this result is attributed to Rosenlicht.}
\item The transcendental degree of  $\C(\Y)^G$  equals to $\displaystyle{\dim \Y- \max_{p\in \Y}\dim Gp}$.
\item If  the field $\C(\Y)$ is rational\footnote{i.e. isomorphic to a field of rational functions of a finite number of independent variables.} and the  transcendental degree of  $\C(\Y)^G$  over  $ \C$ equals to 1 or 2, then  $\C(\Y)^G$  is rational over $\C$.\footnote{In  \cite{vinberg89}, this result is attributed to L\"uroth and Castelnuovo.}
\end{enumerate}
\end{proposition}
%%%
\begin{remark}\label{rem-sep-gen-real} It is worthwhile mentioning that the second part of the proposition is not valid over real numbers.  For example, the field of rational invariants for the action of the group $\R^*$ (non-zero real  numbers under multiplication) on $\R^2$ defined by $(x,y)\mapsto(\lambda^2 x,\lambda^2 y)$ is generated by $K=\frac{x}{y}$, but $K$ is not separating. Conversely,  for the translation action of $\R$  on $\R^2$ defined by $(x,y)\mapsto( x+a,y)$, the invariant $K=y^3$ is separating but not generating.
\end{remark}

%%%% %%
\subsection{Equivalence classes and symmetries of algebraic curves }\label{Sec:Curves}
%%%%%%
We now restrict our attention  to regular actions of  algebraic groups on the complex projective plane $\C\PP^2$. Such an action  induces a homomorphism from $G$ to $Aut(\C\PP^2)=\PGL(3)$, see \cite{Ha77}. Thus we view an algebraic group $G$ acting on $\C\PP^2$ as a closed subgroup of the projective linear group $\PGL(3)$\footnote{From now on we will refer to $\PGL(3)$ as the projective group.}. An element $g\in G$  can be represented by a $3\times 3$ non-singular complex matrix $A_g$, which is defined up to scaling.  
 We use homogeneous coordinates $[x_0:x_1:x_2]$ to represent a point $\pp\in  \C\PP^2$. Then the action of $G$  on $\C\PP^2$ is  defined by:
\beq\label{p-act} g\cdot \p=[\phi_0(g,\pp): \phi_1(g,\pp): \phi_2(g,\pp)], \text{ where }
\left[
\begin{array}{c}
\phi_0(g,\pp)  \\
\phi_1(g,\pp)  \\
\phi_2(g,\pp)
\end{array}
\right]
=
A_g\, 
\left[
\begin{array}{c}
x_0  \\
x_1  \\
x_2
\end{array}
\right].
\eeq
On $\C^2$, we use coordinates $(x,y)$. For an affine point $p=(x,y)\in \C^2$, we  use an abbreviation  $[1:p]=[1:x:y]$ to denote the corresponding projective point. The action \eqref{p-act} induces  a  rational action  $\Phi:G\times \C^2\dasharrow\C^2$ given by 
\beq\label{a-act}
g\cdot p= \left(\frac{\phi_1(g,[1:p])}{\phi_0(g,[1:p])}, \frac{\phi_2(g,[1:p])}{\phi_0(g,[1:p])}\right).
\eeq

We are interested  in the characterization of the equivalence classes of algebraic curves under this action. 
Given a curve $\Chi\subset \C^2$, let $g\cdot \Chi$ denote the the image of $\Chi$ under $g$, namely ${\Phi(g,X)}$.
As this is a rational action, the image may not be an algebraic curve, and so we  will consider its Zariski closure $\overline{g\cdot X}$. 
\begin{definition}
We say that an  algebraic curve $\Chi\subset \C^2$ is \emph{$G$-equivalent} to an  algebraic curve $Y \subset \C^2$ if there exists $g\in G$ such that $\Chi=\overline{g\cdot Y}$.
\end{definition}
Clearly $G$-equivalence satisfies all properties of an equivalence relation, and we  use the notation $X\underset G{\cong} Y$  to denote the $G$-equivalence of curves $X$ and $Y$.
Elements $g\in G$ defining self-equivalences of $\Chi$ are called symmetries of $\Chi$ in $G$.  It is not difficult to show that the set of all symmetries
$$
\SymG=\{g\in G\  |\ \Chi=\overline{g\cdot \Chi} \}
$$
form a closed algebraic subgroup of $G$,   called the \emph{symmetry group} of $\Chi$ with respect to $G$.

Note that the symmetries of $\Chi$ that fix every point of the curve form a normal subgroup of $\SymG$, called the \emph{stabilizer group} of $\Chi$ with respect to $G$:
$$
\StabG=\displaystyle{\bigcap_{p\in\Chi}} G_p.
$$
We show that for a natural class of curves,  $\StabG$ only consists of the identity element. 
\begin{proposition}\label{Thm:SymEff}
For an irreducible curve $\Chi\subset \C^2$ of degree greater than one, the stabilizer group $\StabG$ consists of only the identity.
\end{proposition}
%%%%
\begin{proof}
For $g\in G$ and let  $A_g\in \mathcal{GL}(3)$ be any of its representatives. 
Then a point $p\in \C^2$ is fixed by $g$ if and only if 
$(1,p)$ is an eigenvector of $A_g$. 
Therefore, the set $\C^2_g$ of points fixed by $g$ is the intersection of the affine plane $\{x_0=1\}$  with the union of the eigenspaces of the matrix $A_g$. 
There are three possibilities: (1)  $A_g$ has three linearly independent eigenvectors, then $\C^2_g$ consists of at most\footnote{``At most'' because an eigenspace may be parallel to the  $\{x_0=1\}$ plane.}  three distinct points, (2)   $A_g$ has an eigenspace of dimension 2 and   an eigenspace of dimension 1, then  $\C^2_g$ consists of at most a line and a point,  (3)  $A_g$ has an eigenspace of dimension 3, then $\C^2_g=\C^2$.

If  $g\in\StabG$, then  $\Chi\subset \C^2_g$. Since $\Chi$ is irreducible of degree $>1$, it follows that $\C^2_g=\C^2$.
This implies that $A_g$ is a scalar multiple of the identity matrix and $g$ is the identity element of $\PGL(3)$. 
\end{proof}
%%%%

We finish this section by proving two useful propositions concerning  the orbits of  $\Sym(\Chi,G)$.
%%%%
\begin{proposition}\label{pr:InfOrbit}
If $\Chi$ is irreducible of degree greater than one, then $|\Sym(\Chi,G)|$ is infinite if and only if  there exists a point $p\in \Chi$  whose orbit under $\Sym(\Chi,G)$ is dense in $X$.  
\end{proposition}
\begin{proof} Let   $H= \Sym(\Chi,G)$. This is an algebraic group acting on $\Chi$.

($\Rightarrow$) Assume $|H|$ is infinite. Then since $H$ is algebraic, $\dim  H>0$.  Let $H^0$ denote the connected component of $H$ containing $e$. By   \cite[Prop. 2.2.1]{springer89}, this is a closed normal subgroup of $H$ of finite index and so  $\dim  H^0>0$. By Proposition~\ref{action-basics}, for any $p\in X$ the orbit $H^0p$ is  an irreducible quasi-affine subvariety of $\Chi$. Since  $\dim \Chi=1$, the dimension of $H^0p$ is either zero or one.
If for all $p\in X$, $\dim H^0p=0$, then  $H^0p=\{p\}$ for all $p\in X$. In this case, $\StabG$ contains $H^0$, contradicting the statement of Proposition~\ref{Thm:SymEff}. Therefore, there exists $p\in \Chi$ such that $\dim H^0p=1$. Since $X$ is irreducible of dimension 1,  this implies  
$\overline {H^0p}=\Chi$.

($\Leftarrow$) Assume there exists a point $p\in \Chi$  whose orbit under $H$ is dense in $X$. Then $\dim Hp=1$. By Proposition~\ref{action-basics}, $\dim Hp\leq\dim H$. Therefore $\dim H>0$ and so $|H|$ is infinite. 
\end{proof}

%%%
%%%%
\begin{proposition}\label{pr:FOrbit}
If $\Chi$ is irreducible of degree greater than one and $|\Sym(\Chi,G)|=n<\infty$, then for all but finitely many points $p\in \Chi$  the orbit under $\Sym(\Chi,G)$ consists of exactly $n$ distinct points.  
\end{proposition}
\begin{proof} Let   $H= \Sym(\Chi,G)$. For $g\in H$, define $X_g=\{p\in X\,|\,   g\cdot p =p\}$.  From the proof of Proposition~\ref{Thm:SymEff} it follows that if $g\neq e$, then  $X_g$ is either empty or finite.  Consider the set $E_g=\{p\in X\,|\, g\cdot p \text{ is undefined}\}$, which is also empty or finite.
Since $|H|$ is finite, the set
$\Delta=\cup_{g\in H} (E_g\cup X_g)$ is empty or finite. For all $p\in X\backslash \Delta$, $g\cdot p$ is defined for all $g\in H$  and the stabilizer $H_p=\{e\}$.  Then 
$|Hp|=|H|/|H_p|=n.$
\end{proof}
It is  important to note that  under   the action $G\subset  \PGL(3)$ described by \eqref{a-act} the degree and the irreducibility property are preserved. From now on and throughout the paper  we will make the following assumptions:
\begin{assum} \label{assum1} \hfill
\begin{enumerate}
\item  A group $G$ is a closed subgroup of $\PGL(3)$  with $\dim G>0$.
\item  The rational action of $G$ on $\C^2$ is defined by \eqref{p-act} and \eqref{a-act}.
\item $X\subset \C^2$ is an irreducible algebraic curve of degree greater than one.  
\end{enumerate}
\end{assum}

%%%%%
\subsection{Classifying differential invariants}\label{Sec:Inv}
%%%

To  define differential invariants, we  prolong the action of $G$ to the jet space $J^n$ of planar curves. For our purposes, we can ignore the points where the curve has  vertical tangent and identify 
$J^n$ with $\C^{n+2}$. The coordinate functions on  $J^n$ are denoted by  $(x,y,y^{(1)},\hdots , y^{(n)})$.
Although formally,  $y^{(k)}$ is  viewed as  an independent coordinate function, we define the prolongation formulas keeping in mind  that 
  $y^{(k)}$ is the ``place holder'' for the $k$-th derivative  of $y$ with respect to $x$.
  
\begin{definition}\label{def-prolong} Let $G$  act on  $\C^2$. For $g\in G$, let  $(\overline{x},\overline{y})=g\cdot (x,y)$.    \emph{The prolongation of the $G$-action from $\C^2$ to $J^n$} is a rational  action defined by
$$
g\cdot (x,y,y^{(1)},\hdots , y^{(n)})=(\overline{x},\overline{y},\overline{y}^{(1)},\hdots , \overline{y}^{(n)})
$$
where
$$
\overline{y}^{(1)}= \frac{\frac{d}{dx}\left[ \overline{y}(g,x,y)\right]}{\frac{d}{dx}\left[ \overline{x}(g,x,y)\right]}
\text {\,\, \ and\,\, \ }
 \overline{y}^{(k+1)}= \frac{\frac{d}{dx}\left[ \overline{y}^{(k)}(g,x,y,y^{(1)},\hdots,{ y^{(k)}})\right]}{\frac{d}{dx}\left[ \overline{x}(g,x,y)\right]} 
 \text{  for } k=1,\hdots,n-1. $$
The operator $\frac{d}{dx}$ is the {\it  total derivative operator}. This is the unique $\C$-linear operator mapping $C(J^n)\to C(J^{n+1})$ for $n\geq 0$ satisfying the product rule, $\frac{d}{dx}(x)=1$, and $\frac{d}{dx} (\jy k)=\jy{k+1}$ for $k\geq 0$.
Here we use the convention that $y=\jy 0$ and coordinate functions of $g$ are considered to be constant with respect to $x$. 
\end{definition}
%%%
\begin{definition} %
A rational function $K(x,y,y^{(1)},\hdots,{ y^{(n)}})$ on $J^n$ is called a {\it rational   differential function}. The {\it  differential order} of $K$ is the maximal $k$, such that $K$ explicitly depends on $\jy k$:
$$\ord(K)=\max_{i}\left\{i\,\Big |\, \pd{K}{\jy i}\neq 0\right\}.$$  
If $K$ is invariant under the prolonged  action it is called a  {\it rational differential invariant}.

\end{definition}
%%%
Note that if $\ord(K)=k$, then $K\in \C(J^n$) for all $n\geq k$.  In Theorem~\ref{th-gen}, we show that the  field  $\C(J^r)^G$ of rational invariants of the order at most $r=\dim G$  has a very simple structure. We start by formulating (in our context) an important result originally due to Ovsiannikov \cite{ovsiannikov78} (see also \cite[Theorem 5.11]{olver:purple}).\footnote{We stated this result under Assumptions~\ref{assum1} given at the beginning of the section. For general actions of algebraic groups on algebraic varieties one needs to assume local effectiveness of the action (the set of elements  in $G$ with a trivial action is finite). The theorem was originally stated  for Lie groups acting on  smooth (non-algebraic) real manifolds, and in this setting, as was shown in \cite{olver00},  a stronger assumption of local effectiveness on all open subsets is required. The proof remains valid over $\C$.}  

%%%
\begin{proposition}\label{prop-ovs}
  Let a group $G$ of dimension $r$ act on $\C^2$. Then there is  $k\geq 0$ such that, for all $n\geq k$, the maximal orbit dimension  of the prolonged action on $J^n$  is $r$.
\end{proposition}
%%%
   
%
We need the following  lemma that immediately follows from  \cite[Prop.~5.15]{olver:purple} and the fact that two rational functions are algebraically independent if and only if at a generic point their gradients are linearly independent. This fact is not difficult to prove and for polynomial functions 
is known as the Jacobian criterion of independence. We leave its proof to the reader. 
%%%
\begin{lemma}\label{lem-dk-dh}
Assume $K_1$ and $K_2$ are two algebraically independent
rational differential invariants, 
such that  $\max\left\{\ord(K_1),\ord(K_2)\right\}=k$. Then
$$ \frac {dK_2}{dK_1}:=\frac{\frac {dK_2}{dx}}{\frac {dK_1}{dx}}$$
is a rational differential invariant of order $k+1$.
\end{lemma}
%%%%
%%%
The proof of the next theorem  invokes the line of the argument  in the proof of Theorem~5.24 in \cite{olver:purple} in combination with  Proposition~\ref{inv-field-prop} stated above.  
%%%%
\begin{theorem}\label{th-gen}  Let $\dim G=r$, then the field of $\C(J^r)^G$ of rational invariants on $J^r$ is a rational field of transcendental degree two. In other words, there exists two  rational invariants $K_1$ and $K_2$ such that 
\beq\label{eq-jr} \C(J^r)^G=\C(K_1,K_2).\eeq 
Moreover  $K_1$ and $K_2$ can be chosen so that $K_1$ is of differential order $k$, strictly  less  than $r$, and $K_2$ is of differential order $r$. In addition,  the field $\C(J^k)^G$ of rational invariants on $J^k$ is a rational field of transcendental degree one and
\beq\label{eq-jk}  \C(J^k)^G=\C(K_1).\eeq 
\end{theorem}

\begin{proof}  The dimension of an orbit can not exceed the dimension of the group. Therefore, since   $\dim J^{r-1}=r+1$, the transcendental degree of  $\C(J^{r-1})^G$ is at least 1 by Part 3.~of  Proposition~\ref{inv-field-prop}. Thus there exists a rational  invariant $K_1$  such that $\ord(K_1)=k_1<r$.  We may assume that the order $k_1$ of $K_1$ is minimal among all such invariants. Similarly, since  $\dim J^{r}=r+2$, the transcendental degree of  $\C(J^{r-1})^G$ is at least 2, and there exists a rational invariant $K_2$, algebraically independent from $K_1$,   such that $\ord(K_2)=k_2\leq r$.  By  the minimality assumption on $k_1$, we have $k_1\leq k_2$. Assume that  $k_2<r$.  By Proposition~\ref{lem-dk-dh}, invariant $H_1=\frac {dK_2}{dK_1}$ is of order $k_2+1$. For $i>1$, we   define invariants $H_i=\frac {dH_{i-1}}{dK_1}$.
The $n+2$  invariants $K_1,K_2, H_1, H_2, \dots, H_n$ are of orders $k_1, k_2, k_2+1, \dots k_2+n$, respectively.  Since $K_1$ and $K_2$  are independent, and each  subsequent invariant contains a new jet variable,  the gradients of these invariants as functions on $J^{k_2+n}$ are independent, and hence the invariants are independent. Therefore the maximal orbit dimension on $J^{k_2+n}$  does not exceed $\dim J^{k_2+n}-(n+2)=k_2$.  Since $n$ can be arbitrary large, it follows  
from Proposition~\ref{prop-ovs} that $k_2=r$. In summary, we proved so far
$$ k_1<k_2=r$$ and  that there are no differential invariants of orders strictly less than $k_1$, or strictly between $k_1$ and $r$. 

Assume that there is an invariant $K_3$ of order $r$, independent of $K_1$ and $K_2$. Then  by similar argument as in the above paragraph, the  $n+3$  invariants $K_1,K_2, K_3, H_1, H_2, \dots, H_n$  of orders $k_1, r, r, r+1, \dots r+n$, respectively,  are independent for all $n$. It follows that
the maximal orbit dimension on $J^{r+n}$  does not exceed $\dim J^{r+n}-(n+3)=r-1$ for all $n$. This contradicts Proposition~\ref{prop-ovs}.

We conclude that the transcendental degree of  $\C(J^k)^G$ is 1 and  the transcendental degree of  $\C(J^r)^G$ is 2.  Then \eqref{eq-jr} and 
\eqref{eq-jk} follow from  Part 4~of  Proposition~\ref{inv-field-prop}.
\end{proof}

%%%%%%%%%%%%%%%

\begin{remark} In fact,  from Theorem~5.24 in \cite{olver:purple} and  Sophus Lie's classification of  all infinitesimal group actions on the plane (see Table~5 in  \cite{olver:purple}) it follows that there  are only three possibilities for the differential order $k$ of  the lower order classifying invariant $K_1$, namely $k=r-1$, $k=r-2$ and $k=0$.
For most of the  actions (and all  actions  considered in Section~\ref{Sec:Projective} of this paper) $k=r-1$.  The case $k=0$ occurs if and only if  the action $G$ is intransitive on $\C^2$. An example of such action is the action of a $2$-dimensional subgroup of $\PGL(3)$, given by $(x,y)\to (\lambda x+a, y)$, where $\lambda\in \C^*$ is non-zero and $a\in \C$.  Among subgroups of $\PGL(3)$, the third possibility,  $k=r-2\neq 0$, occurs   only for two actions: (1) a three-dimensional subgroup acting by  $(x,y) \to (\lambda x+a,  \lambda y+b)$, where  $\lambda\in \C^*$ and $a, b\in \C$ and (2) a four-dimensional subgroup acting by  $(x,y)\to (\lambda x+a, cx+\lambda^2 y+b)$, where  $\lambda\in \C^*$ and $a, b,c\in \C$
 \end{remark}
 %%%%%%%%
 We can  use the same   definition of the classifying invariants as was given in \cite[Definition 7]{BKH} in the real case.
\begin{definition}
\label{def-dsep}
Let an $r$-dimensional algebraic group $G$ act on $\mathbb{C}^2$.
Let $K_1$ and $K_2$ be \emph{rational differential invariants} of orders ${k<r}$ and $r$,
respectively.
The set ${\mathcal I}=\{K_1,K_2\}$ is called {\em classifying } if $K_1$ separates orbits on a nonempty Zariski-open subset $W^{k}\subseteq J^{}$ and ${\mathcal I}$ separates orbits on a nonempty Zariski-open
subset $W^r\subseteq J^{r}$.
\end{definition}
%%%%%%%
Over $\C$ we can prove existence of a classifying set of invariants of any group action:

%%%
\begin{theorem} \label{th-class} For any action of $G\subset\PGL(3)$ on $\C^2$ there exists a classifying set ${\mathcal I}=\{K_1,K_2\}$ of differential invariants. Moreover the set $\mathcal I$ is classifying if and only if  $\I$   generates the field  $\C(J^{r})^G$ of rational differential  invariants of order  at most $r=\dim G$ and $K_1$ generates the field $\C(J^{r-1})^G$ of rational invariants of order  at most $r-1$ .
\end{theorem}
%%%
%%%%%
\begin{proof}  This result follows immediately from Theorem~\ref{th-gen} and Part 2~of Proposition~\ref{inv-field-prop}.
\end{proof}
%%%
Remark~\ref{rem-sep-gen-real} underscores that  over $\R$ the above proof of Theorem~\ref{th-class} is not valid.  It is an interesting question, whether or not  the statement of this theorem (or possibly some modification) is valid over $\R$. 
In Section~\ref{Sec:Sig} we show that signatures based on classifying invariants characterize the equivalence classes of generic algebraic  curves. In Section~\ref{subsec:Groups} we list  classifying sets of invariants for the full projective group and several of its classical subgroups.
 The following propositions asserts a simple relationship between any two classifying sets of invariants.
%%%
\begin{proposition}\label{prop-choice} Let  ${\mathcal I}=\{K_1,K_2\}$ be a classifying sets of differential invariants for  the  
 action of $G$ on $\C^2$.  Let   $\tilde{\mathcal I}=\{\tilde K_1,\tilde K_2\}$ be another pair of differential invariants. Then  $\tilde{\mathcal I}$ is a classifying set  if and only if  there exist constants $a,b,c,d\in \C$, such that   $ad-bc\neq 0$ and rational functions $\alpha,\beta,\gamma, \delta\in \C(\kappa)$, such that $\alpha\delta-\gamma\beta\neq 0$ such that
 \beq\label{class-rel}\tilde K_1=\frac{aK_1+b}{cK_1+d}\text{  and  } \tilde K_2=\frac{\alpha(K_1)K_2+\beta(K_1)}{\gamma(K_1)K_2+\delta(K_1)}.\eeq 
\end{proposition}
%%5

%%%
\begin{proof}
By  Theorems~\ref{th-gen}   and~\ref{th-class}, we know that $\C(J^k)^G=\C(K_1)$ and  $\C(J^r)^G=\C(K_1,K_2)$ are  rational fields of transcendental  degrees 1 and 2 respectively for $r=\dim G$ and some integer $k<r$.
Moreover, from the proof of Theorem~\ref{th-gen}, we know that there are no differential invariants of order strictly  greater than $k$ and strictly less than $r$.

 Assume first that   $\tilde{\mathcal I}$  is a  classifying set. Then for $r=\dim G$ and some integer $k<r$, we have $\ord(\tilde K_2)=r $ and $\ord(\tilde K_1)=k$ and $\C(J^k)^G=\C(\tilde K_1)$ and  $\C(J^r)^G=\C(\tilde K_1,\tilde K_2)$. 
Now we have two sets of generators for each of the fields $\C(J^k)^G$ and  $\C(J^r)^G$ and so there  exist invertible rational functions $\Phi\in \C(\kappa_1)$ and $\Psi\in \C(\kappa_1,\kappa_2)$   such that  $\tilde K_1=\Phi(K_1)$ and $\tilde K_2=\Psi(K_1,K_2)$.  The function $\Phi$  induces an automorphism of  $\C(J^k)^G$ fixing $\C$. 
  It is known (see, for instance,  \cite[Exercise 6, Sec. V.2]{Halgebra}) that the only automorphisms of a rational field $\mathbb K (z)$ fixing the ground field $\mathbb K$ are given by linear fractional maps over  $\mathbb K$.  The first formula in \eqref{class-rel} follows with  $\mathbb K=\C$.  
Similarly $\Psi$ induces an automorphism of   $\C(J^r)^G=\C(J^k)^G(K_2)$ fixing $\C(J^k)^G=\C(K_1)$.  By the  same argument, with  $\mathbb K=\C(K_1)$, the second formula in \eqref{class-rel} follows.   

Now assume that $\tilde K_1$ and $\tilde K_2$ are given by \eqref{class-rel}. Then since these formulas are invertible, $\ord(\tilde K_1)=k $, $\ord(\tilde K_2)=r$, and  $\C(J^k)^G=\C(\tilde K_1)$, while  $\C(J^r)^G=\C(\tilde K_1,\tilde K_2)$. By Theorem~\ref{th-class}, $\tilde{\mathcal I}=\{\tilde K_1,\tilde K_2\}$ is a classifying set.
\end{proof}
%%%%

%%%%%
\subsection{Restriction to algebraic curves}\label{Sec:DiffCurve} 
%%%

To evaluate differential functions on an affine curve, we lift the curve into the jet space as follows. 
Let  $F(x,y)\in \C[x,y]$ be irreducible and $\Chi = V(F) \subset \C^2$. 
For any point $p=(p_1, p_2)\in X$ with $F_y( p)\neq 0$ the curve $X$ agrees in some neighborhood of $ p$ with the graph of an analytic function $y=f(x)$. 
Then for a positive integer $n$, we  can define $y^{(n)}_\Chi( p)=f^{(n)}(p_1)$ to be the  $n$-th derivative of $f(x)$ at $x=p_1$. 
One can  show that for each $n\in \Z_+$,  $y^{(n)}_X$  is a rational function on $\Chi$ that, using  the implicit differentiation,  can be written as a rational function of partial derivatives of $F$. For example, 
\begin{equation}\label{Eq:PartialDeriv}
y_\Chi^{(1)}=\frac{-F_x}{F_y} \ \ \ \text{ and } \ \ \ \  y_\Chi^{(2)}=\frac{-F_{xx}F_y^2+2F_{xy}F_xF_y-F_{yy}F_x^2}{F_y^3}.
\end{equation}
It follows  that, $y^{(n)}_X$ is a rational function on $X$.

\begin{definition}
The \emph{$n$-th jet} of a curve $\Chi\subset \C^2$, denoted $X^{(n)}$, is {the algebraic closure of the image of $X$ under} the rational map $j^n_\Chi: \Chi\dashrightarrow J^n$, where for $p\in \Chi$,
$$
j^n_{\Chi}(p)=(x(p),y(p),y^{(1)}_\Chi(p),\hdots,y^{(n)}_\Chi(p)).
$$
\end{definition}
Note that the prolongation of the action of $G$ to $J^n$ (Definition~\ref{def-prolong}) 
is defined so that the following fundamental property holds: 
\beq\label{pr-g-comm}
j^n_{g\cdot\Chi}(g\cdot p)\ = \ g\cdot j^n_\Chi(p) \ \text{ for all $g\in G$ and $p\in \Chi$ where $g\cdot p$ is defined.}
\eeq
In particular, the $n$-th jet of the image of $\Chi$ under the action of $g\in G$ coincides with the  image of  the $n$-th jet of $\Chi$ under the prolonged action of $g$:
\beq\label{pr-gX} \overline{g\cdot X^{(n)}}= \overline{(g\cdot X)^{(n)}}.\eeq

\begin{definition}
For a curve $\Chi$, the \emph{restriction} of a differential function $K$ to $\Chi$ is denoted $K|_\Chi$ and defined by the composition,
$K|_\Chi=K\circ j^n_\Chi.$
\end{definition}
If $K$ is a \textit{rational} differential function on $J^n$, then $K|_\Chi$ is a rational function on $X$, and 
we can obtain the explicit formula for $K|_\Chi$ as a rational function of $x$ and $y$ by substituting the expressions $y_\Chi^{(1)},\hdots, y^{(n)}_\Chi$ in \eqref{Eq:PartialDeriv} for coordinates $y^{(1)},\hdots, y^{(n)}$.
%%%%
\begin{definition}\label{Def:IReg}
Let $\mathcal{I}=\{K_1,K_2\}$ be a classifying set of rational differential invariants for a group  $G$ of dimension $r$.  Let $\ord(K_1)=k$ and let  $W_1\subset J^k$ be  a maximal domain of separation for $\{K_1\}$ and $W_2\subset J^r$ be a maximal domain of separation for $\mathcal{I}$.  Then, for $\Chi\subset \C^2$, a point $p\in\Chi$ is called \emph{$\mathcal{I}$-regular} if
\begin{itemize}
\item[(a)] $j^r_\Chi(p)$ is defined;
\item[(b)] $j_\Chi^{k}(p)\in W_1$ and $j^r_\Chi(p)\in W_2;$
\item[(c)] $\frac{\partial K_1}{\partial y^{k}}|_{j_\Chi^{k}(p)}\neq 0$ if $K_1$ is constant on $\Chi$, and $\frac{\partial K_2}{\partial y^{(r)}}|_{j^r_\Chi(p)}\neq 0$ otherwise.
\end{itemize}
\end{definition}

The condition that $j^r_\Chi(p)$ is defined can equivalently be stated as $F_y(p)\neq 0$ where $F(x,y)$ is the polynomial whose zero set equals $\Chi$. Thus singular points of $\Chi$ are not $\mathcal{I}$-regular.

\begin{definition}\label{Def:NonExcept} A complex algebraic curve $\Chi\subset \C^2$ is called \emph{non-exceptional} with respect to a classifying set of differential invariants, $\mathcal{I}$, if all but a finite number of its points are $\mathcal{I}$-regular. 
\end{definition}

%%%%

%CONJECTURE:  a generic curve is non exceptional if for  all choices of  $G$-classifying set of differential invariants%%%%
%%%
We will need the following lemma to show that generic curves are non-exceptional. 
%%%%
\begin{lemma}\label{lem:JetGen}
Let $d,n$ be positive integers satisfying $n\leq \binom{d+2}{2}-2 $. 
For a generic point $a = (a_0, \hdots, a_n)\in \C^{n+1}$, there exists an algebraic curve $X \subset \C^2$ of degree $d$ 
for which $(0,a_0) \in X$ and $j^{(n)}_X(0,a_0) \ = \ (0,a_0,\hdots, a_n)$. 
\end{lemma}
%%%%
\begin{proof}
Consider the subset $\mathcal{Y}$ of $\PP(\C[x,y]_{\leq d}) \times \C^{n+1}$
consisting of pairs $([F],a)$ for which $F$ is irreducible of degree $d$,  
$F(0,a_0)=0$, $F_y(0,a_0)\neq 0$, and $j^{(n)}_{V(F)}(0,a_0) = (0,a_0,\hdots, a_n)$. 
Since $j^{(n)}_{V(F)}$ is a rational function of both the points of $V(F)$ and the coefficients 
of $F$, as seen in \eqref{Eq:PartialDeriv}, this is a quasi-projective variety. 
The conditions $F(0,a_0)=0$ and $a_k = y^{(k)}_X(0,a_0)$ are algebraically independent, 
since each involves a new variable, $a_k$. 
From this, it follows that $\Y$ has codimension $n+1$ in $\PP(\C[x,y]_{\leq d}) \times \C^{n+1}$ 
and thus dimension $\binom{d+2}{2}-1$.  
The projection of $\mathcal{Y}$ onto $\C^{n+1}$ is therefore a quasi-affine variety.
It either contains a nonempty Zariski-open set or is contained in a hypersurface in $\C^{n+1}$. 
We need to rule out the latter when $n\leq\binom{d+2}{2}-2$.

Suppose for the sake of contradiction that for some $n\leq\binom{d+2}{2}-2$, 
there is a polynomial relation $P(y,y^{(1)}, \hdots, y^{(n)})=0$ that holds for every point 
on the image of $X\cap V(x)$ under $j^{(n)}_X$ for every irreducible curve $X$ of degree $d$. 
Without loss of generality, we can assume that $n$ is the minimal integer for which this holds 
and that the polynomial $P$ is irreducible. 
Then, by Bertini's theorem, for generic $a_0, \hdots, a_{n-1} \in \C$, $P(a_0, \hdots, a_{n-1},y^{(n)})$ 
is a non-zero polynomial in $y^{(n)}$ with simple roots, around which $y^{(n)}$ is an analytic function of 
$a_0,\hdots, a_{n-1}$. 
Due to  the uniqueness  theorem for the solutions of complex ODEs \cite{In44}, for any  such  $a_0, \hdots, a_{n-1}$ and $a_n$ with $P(a_0, \hdots, a_n)=0$, 
there exists a unique solution $y=f(x)$ to the differential equation $P(y, y^{(1)}, \hdots, y^{(n)})=0$ 
satisfying the initial conditions $x=0$, $f(0)=a_0$, and $f^{(k)}(0)=a_k$ for $k=1, \hdots, n$. 

If there exists an irreducible polynomial $F\in \C[x,y]$ of degree $d$ for which 
$F(x,f(x))$ is identically zero, then $F$ is unique up to scaling. 
This means that every point in the projection of $\Y$ onto $\C^{n+1}$ has at most one preimage. 
Since the projection has dimension $\leq n$, this implies that 
the dimension of $\Y$ is also at most $n$, which contradicts the calculation that $\dim(\Y) $ equals $ \binom{d+2}{2} -1 >n$. 
Therefore the projection of $\Y$ onto $\C^{n+1}$ must be Zariski-dense. 
\end{proof}
%%%%%
\begin{theorem}\label{thm-non-excep-generic} 
Let $\I$ be a $G$-classifying set of rational differential invariants 
for the action of a group $G$.
Then for $d\in \Z_+$ with $\binom{d+2}{2} -2\geq \dim G$, 
a generic plane curve of degree $d$ is non-exceptional with respect to $\I$. 
\end{theorem}
%%%%%%
\begin{proof}
For an irreducible curve $\Chi$, the $\I$-regular points form a Zariski-open subset of $X$, as seen in Definition~\ref{Def:IReg}.  
Either this is all but finitely-many points of $X$, in which case $X$ is non-exceptional, or empty, meaning that 
no points of $X$ are $\I$-regular. In particular, if all intersection points of $X$ with $V(x)$ are $\I$-regular, then $X$ is 
non-exceptional. 

Indeed, the condition that a point $p$ is $\I$-regular on $X$ is equivalent to 
the jet $j^{(r)}_X(p)$ belonging to a Zariski-open subset $\mathcal{U}$ of $J^{r}\cong \C^{r+2}$, 
where $r = \dim(G)$. 
Consider the quasi-projective variety $\Y$ defined in the proof of Lemma~\ref{lem:JetGen} with $n=r$. 
Its intersection with $\PP(\C[x,y]_{\leq d})\times \mathcal{U}$ is an open subset of $\Y$, which 
is nonempty by Lemma~\ref{lem:JetGen}. 

Furthermore, the projection of $\Y$ onto  $\PP(\C[x,y]_{\leq d})$ is dominant (i.e. the image in Zariski-dense). 
Specifically, consider the open dense set of irreducible polynomials $F\in \C[x,y]_{\leq d}$
 for which $F(0,y)$ has a simple root $y=a_0$ at which $F_y(0,a_0)$ is nonzero. 
 For any such $F$,  $([F],a)$ belongs to $\Y$, where $j^{(r)}_{V(F)}(0,a_0) = (0,a)$. 
 It follows that the projection of the set $\Y\cap ( \PP(\C[x,y]_{\leq d})\times \mathcal{U}) $ 
onto $\PP(\C[x,y]_{\leq d})$ is also dominant. 
Therefore, for a generic plane curve of degree $d$, the points  $X\cap V(x)$
are $\I$-regular in $X$, and thus $X$ is non-exceptional. 
\end{proof}
%%%%
We will also make use of the $G$-invariance of the set of  non-exceptional curves.
%%%
\begin{lemma} \label{lem-non-except} If $X$ is non-exceptional then so is $Y=\overline{g\cdot X} $ for all $g\in G$.
\end{lemma}

\begin{proof} We check that if conditions (a) -- (c) in Definition~\ref{Def:IReg} are satisfied  by all but finitely many points on $X$, then the same is true for $Y$.

(a)  Assume that there are  at most  finitely many points  $p\in X$, such that  $j^r_\Chi(p)$ is undefined (equivalently $F_y(p)= 0$, where $F$ is a defining polynomial of $X$). This is, in fact, true for any irreducible curve of  degree greater than 1.  Since the action of $G$  preserves these properties, there are   at most  finitely many points $p\in Y$, such that $j^r_Y(p)$ is undefined. 

(b) Assume that there are at most  finitely many points $p\in X$, such that $j_\Chi^{k}(p)\notin W_1$ and $j^r_\Chi(p)\notin W_2$. From the $G$-invariance of $W_1$ and $W_2$ and \eqref{pr-g-comm}, combined with the fact that $Y\backslash  ({g\cdot X})$ is a finite set,   it follows that there  are at most  finitely many points $p\in  Y$ such that $j_Y^{k}(p)\notin W_1$ and $j^r_Y(p)\notin W_2$.

(c)  We start by showing that if $K$ is a differential invariant of order $n$, then 
the set of points ${p^{(n)}}\in J^n$ where $\frac{\partial K}{\partial y^{(n)}}(p^{(n)})\neq 0$ is $G$-invariant.
  Since $K$ is invariant,   $K(p^{(n)})=K(g\cdot p^{(n)})$, whenever both sides are defined, and the differentiation with respect $y^{(n)}$using the chain rule yields:
\begin{align*}
&\frac{\partial K}{\partial y^{(n)}}\!\left(p^{(n)}\right)\\
&=\frac{\partial K}{\partial \overline{x}}\!\left(g\cdot p^{(n)}\right)\frac{\partial \overline{x}}{\partial y^{(n)}}\!\left(p^{(n)}\right)+\frac{\partial K}{\partial \overline{y}}\!\left(g\cdot p^{(n)}\right)\frac{\partial \overline{y}}{\partial y^{(n)}}\!\left(p^{(n)}\right)+\hdots +\frac{\partial K}{\partial \overline{y}^{(n)}}\!\left(g\cdot p^{(n)}\right)\frac{\partial \overline{y}^{(n)}}{\partial y^{(n)}}\!\left(p^{(n)}\right)\\
&=\frac{\partial K}{\partial \overline{y}^{(n)}}\!\left(g\cdot p^{(n)}\right)\frac{\partial \overline{y}^{(n)}}{\partial y^{(n)}}\!\left(p^{(n)}\right).
\end{align*}
The last equality follows from the fact that the functions $\overline{x}, \overline{y},$ and $\overline{y}^{(i)}$, given in Definition~\ref{def-prolong}, do not depend on $y^{(n)}$ for $i=1,\hdots,n-1$. Thus if $\frac{\partial K}{\partial y^{(n)}}(p^{(n)})\neq 0$, so does every point in the orbit of $p^{(n)}$.

Condition (c) states that, if $K_1$ is constant on $\Chi$, then  for all but finitely many $p\in X$, $\frac{\partial K_1}{\partial y^{k}}|_{j_\Chi^{k}(p)}\neq 0$, otherwise  for all but finitely many $p\in X$, $\frac{\partial K_2}{\partial y^{r}}|_{j_\Chi^{r}(p)}\neq 0$, where $k=\ord(K_1)$ and $r=\ord(K_2)$. Due to \eqref{pr-g-comm}, and $G$-invariance property showed above, the same is true for $Y$. 
\end{proof}

\subsection{Differential signatures of algebraic curves}\label{Sec:Sig} 
In this section, we define the signature map and signature curve and show that signatures characterize the equivalence classes of generic algebraic curves.  Throughout this section, we assume $G$ is an algebraic group with $\dim(G)=r$ and that $\{K_1, K_2\}$ 
are a classifying set of differential invariants with $\ord(K_1) = k < r = \ord(K_2)$, as described above. 

%%%%
\begin{definition}\label{def:sig}
Let $\I=\{K_1,K_2\}$ be a classifying set of rational differential invariants with respect to the action $G$, and let $\Chi\subset \C^2$ be a 
non-exceptional curve.
Then the rational map $\s_{\Chi}:\Chi\dashrightarrow \C^2$ with coordinates $(K_1|_{\Chi},K_2|_{\Chi})$ is called the \emph{signature map}.
The image of $\imS_{\Chi} = \s_{\Chi}(\Chi)$ is called the \emph{signature} of $\Chi$. 
\end{definition}
%%%%
Note that since $\Chi$ is irreducible, then the closure  $\overline{\imS_{\Chi}}$ is also an  irreducible variety of dimension 0 or 1. If  $\dim(\overline{\imS_{\Chi}})=0$, then it is a single point and, therefore, $\s_{\Chi}$ is a constant  map.   If $\dim(\overline{\imS_{\Chi}})=1$, then it is an irreducible planar curve, which we  call the {\it  signature curve} of $\Chi$.  An irreducible polynomial  vanishing on  $\overline{\imS_{\Chi}}$  is called a  {\it  signature polynomial} and is denoted by   $S_{\Chi}$ and it is unique up to scaling by a non-zero constant. %%%%
%%%%
\begin{proposition}\label{Prop:GroupEquiv1}
Assume that $X, Y\subset \C^2$ are $G$-equivalent and non-exceptional with respect to a classifying set of rational differential invariants $\mathcal{I}=\{K_1,K_2\}$. Then  $\overline{\imS_X}=\overline{\imS_Y}$.
\end{proposition}
%%%
\begin{proof}
If $X$ and $Y$ are $G$-equivalent, then there exists $g\in G$ such that $Y=\overline{g\cdot X}$. 
Due to the fundamental property of prolongation  \eqref{pr-g-comm}, we  have $j^r_Y(q)=g\cdot j^r_X(p)$, for any  $p\in X$ where $q=g\cdot p$ is defined. 
Since $K_1$ and $K_2$ are invariant,  we have 
$$K_1(j^r_X(p))=K_1(j^r_Y(q)) \text{ and } K_2(j^r_X(p))=K_2(j^r_Y(q)).$$
This implies  $\sigma_X(p)=\sigma_Y(q)$. Since $g\cdot p$ is defined for all but finitely many points in $X$ and $g\cdot X$ is dense in $Y$, this implies that $\overline{\imS_X}=\overline{\imS_Y}$.
\end{proof}
%%%%
We will gradually work towards proving the converse of the above statement, and thus   showing that the  signature polynomials characterize the equivalence classes of curves. We will also show  the relationship between the cardinality of the preimage of a generic  point under a signature map  and  the cardinality  of the symmetry group. For both of these results we need several lemmas.
%%%%
\begin{lemma} \label{ode-lemma0}  Let $\I=\{K_1,K_2\}$ be a classifying set of rational differential invariants with respect to the action $G$, and let $\Chi, Z \subset \C^2$ be two 
non-exceptional curves, such that the restrictions of $K_1$ to both curves equal to the same constant function:
$$K_1|_X=K_1|_Z=c.$$
 If there exists $p\in X \cap Z$ such that 
\begin{enumerate}
\item  $j^k_X(p)= j^k_Z(p)$, where $k=\ord(K_1)$,
\item $p$ is not exceptional for $X$,
\end{enumerate}
 then $X=Z$. 
\end{lemma}
%%%
\begin{proof} Since $p$ is non-singular for both $X$ and $Z$,  in some neighborhood of $p$, curves $X$ and $Z$ coincide with the graphs of analytic functions $y=f(x)$ and $y=g(x)$, respectively.
 Both  $y=f(x)$ and $y=g(x)$ are solutions of the differential equation 
\beq\label{eq-K1}K_1(x,y, \jy 1,\dots,\jy  k)=c,\eeq with the same initial condition described by the point $j^k_X(p)= j^k_Z(p)$.
Since $p$ is non-exceptional, $\frac{\partial K_1}{\partial y^{k}}|_{j_\Chi^{k}(p)}\neq 0$, and so using the implicit function theorem,  \eqref{eq-K1} can be rewritten as  $\jy k= H(x,y, \jy 1,\dots,\jy  {k-1})$ in a neighborhood of ${j_\Chi^{k}(p)}$, where $H$ is an analytic function of  the jet coordinates. We can now invoke the uniqueness  theorem for the solutions of complex ODEs \cite{In44} to conclude that $f(x)=g(x)$. Therefore $X$ and $Z$ coincide on a positive dimensional subset. Since they are irreducible $X=Z$.
\end{proof}
%%%%%%
\begin{lemma} \label{ode-lemma} Let $\I=\{K_1,K_2\}$ be a classifying set of rational differential invariants with respect to the action $G$, and let $\Chi, Z \subset \C^2$ be two 
non-exceptional curves with the same signature curves, $\overline{\imS_X}=\overline{\imS_Z}$.  If there exists $p\in X \cap Z$ such that 
\begin{enumerate}
\item  $j^r_X(p)= j^r_Z(p)$,
\item $p$ is not exceptional for $X$
\item if $\dim \imS_X>0$ and $S_X(\kappa_1,\kappa_2)$ is a signature polynomial, then $\frac{\partial S}{\partial \kappa_2}|_{\sigma_X(p)}\neq 0$,
\end{enumerate}
 then $X=Z$. 
\end{lemma}

\begin{proof}  If $\sigma_X$ (and, therefore, $\sigma_Z$) is a constant map, then there exists $c\in\C$, such that $K_1|_X=c$ and $K_1|_Z=c$. Then we are in the situation of Lemma~\ref{ode-lemma0} and the conclusion follows.
Otherwise, $\sigma_X$  and, $\sigma_Z$  define the same signature polynomial
$S_X(\kappa_1,\kappa_2)=S_Z(\kappa_1,\kappa_2):=S(\kappa_1,\kappa_2)$.
 Since $p$ is non-singular for both $X$ and $Z$,  in some neighborhood of $p$, curves $X$ and $Z$ coincide with the graphs of analytic functions $y=f(x)$ and $y=g(x)$, respectively.
Both  $y=f(x)$ and $y=g(x)$ are solutions of the differential equation 
\beq\label{eq-S}S\left(K_1(x,y, \jy 1,\dots,\jy  k),K_2(x,y, \jy 1,\dots,\jy  r)\right)=0 \eeq
with the same initial condition described by the point $j^r_X(p)= j^r_Z(p)$.
By assumption, $\frac{\partial S}{\partial \kappa_2}|_{\sigma_X(p)}$ and $\frac{\partial K_2}{\partial y^{(r)}}|_{j^r_\Chi(p)}$ are both nonzero.
Then using the implicit function theorem,  \eqref{eq-S} can be rewritten as  $\jy r= H(x,y, \jy 1,\dots,\jy  {r-1})$ in a neighborhood of ${j_\Chi^{r}(p)}$, where $H$ is an analytic function of  the jet coordinates. 
As in the previous lemma, we invoke  the uniqueness  theorem for the solutions of  ODEs, to conclude   $X=Z$.
\end{proof}

\begin{lemma}\label{self-equiv} Let $\I=\{K_1,K_2\}$ be a classifying set of rational differential invariants with respect to the action $G$, and let $\Chi$ be a
non-exceptional curve.  Let $p,q\in X$ be two non-exceptional points, such that
\begin{enumerate}
\item  $\sigma_X(p)=\sigma_X(q)$
\item if $\dim \imS_X>0$ and $S_X(\kappa_1,\kappa_2)$ is the signature polynomial, then $\frac{\partial S}{\partial \kappa_2}|_{\sigma_X(p)}\neq 0$. 
\end{enumerate}
 Then  there exists $g\in \SymG$, such that $q=gp$. 

\end{lemma}
\begin{proof} Since, $\sigma_X(p)=\sigma_X(q)$ we have
$$K_1(j^r_X(p))=K_1(j^r_X(q)) \text{ and } K_2(j^r_X(p))=K_2(j^r_X(q)).$$
 Since $\I$ is a separating set,  and $p$ and $q$ are non-exceptional,  there exists
$g\in G$, such that $j^r_X(p) =g\cdot j^r_X(q)$. Consider a curve $Z=\overline{g\cdot X}$. By Lemma~\ref{lem-non-except}, $Z$ is non-exceptional.
Condition $\overline{\imS_X}=\overline{\imS_Z}$ holds due to Proposition~\ref{Prop:GroupEquiv1}.
Due to the fundamental property of prolongation \eqref{pr-g-comm}  we have $j^r_Z(p) =g\cdot j^r_X(q)$. This implies     $p=g\cdot q\in Z$   and  $j^r_Z(p) =j^r_X(p)$.    We  verified that $X$ and $Z$ satisfy all conditions of Lemma~\ref{ode-lemma}. Then $X=Z=\overline{g\cdot X}$ and, therefore $g\in \SymG$.
\end{proof}

\begin{lemma} 
\label{Lem:ZeroDimSig}
Suppose that $\Chi$ is a non-exceptional curve with respect to a classifying set of rational differential invariants $\mathcal{I}=\{K_1,K_2\}$. Then the following are equivalent:

\begin{itemize}
\item[(1)] $K_1|_\Chi$ is a constant function on $X$,
\item[(2)] $H=\Sym(X,G)$ is infinite,
\item[(3)] the signature ${\imS}_{X}$ consists of a single point.
\end{itemize}
\end{lemma}
%%%
\begin{proof}
$ (1)\Rightarrow (2)$  Assume $K_1|_\Chi=c$ is a constant function on $X$.  Fix a non-exceptional point $p$.
We will show that any  non-exceptional point on $\Chi$ belongs to the orbit $Hp$. Since non-exceptional points are dense in $X$,  the conclusion would follow from Proposition~\ref{pr:InfOrbit}.

Let $q$ be a non-exceptional point on X. Then $K_1(j^k_X(p))=K_1(j^k_X(q))=c$ where $k$ equals $\ord(K_1)$.
Since $K_1$ is separating on $J^k$, there exists $g\in G$, such that $j^k_X(p)=g\cdot j^k_X(q)$. Consider a curve $Z=\overline{g\cdot X}$. By Lemma~\ref{lem-non-except}, $Z$ is non-exceptional.
Condition $\overline{\imS}_X=\overline{\imS}_Z$ holds due to Proposition~\ref{Prop:GroupEquiv1}. Therefore $K_1|_Z$ is the same constant function as $K_1|_X$.
Due to the fundamental property of prolongation \eqref{pr-g-comm}  we have $j^r_Z(p) =g\cdot j^r_X(q)$. This implies     $p=g\cdot q\in Z$   and  $j^r_Z(p) =j^r_X(p)$.    We  verified that $X$ and $Z$ satisfy all conditions of Lemma~\ref{ode-lemma0}. Then $X=Z=\overline{g\cdot X}$ and, therefore $g\in H$ and so $q\in Hp$.

$ (2)\Rightarrow (3)$ Let $p$ be a non-exceptional point.   For any $q\in Hp$, there exists $g\in H$, such that $p=g\cdot q$ and $X=g\cdot X$.  If $q$ is non-exceptional, it follows from \eqref{pr-g-comm} that $j^k_X(p)=g\cdot j^k_X(q)$. Since $K_1$ is a differential invariant, $K_1|_X(g\cdot j^k_X(q))=K_1|_X( j^k_X(q))$. Then 
$$K_1|_X(j^k_X(p))=K_1|_X(j^k_X(q)) \text{ for all non-exceptional } q\in H_p. $$

Since $H$ is infinite, from  Proposition~\ref{pr:InfOrbit}, it follows  the orbit $Hp$ is dense in $\Chi$. The set of non-exceptional points is also dense in $\Chi$. Thus  $K_1|_X$ is a constant rational function on a dense subset of $\Chi$ and, therefore, is constant on $\Chi$. 

$ (3)\Rightarrow (1)$ Obvious from the definition  of ${\imS}_{X}$.
\end{proof}

We are now ready to prove the converse of the Proposition~\ref{Prop:GroupEquiv1}.
%%%%
\begin{proposition}\label{Prop:GroupEquiv2} If algebraic curves $X, Y\subset \C^2$ are non-exceptional with respect to a classifying set of rational differential invariants $\mathcal{I}=\{K_1,K_2\}$ under an action of $G$ on $\C^2$ and their signature curves  are equal,  $\overline{\imS_X}=\overline{\imS_Y}$, then $X$ and $Y$ are $G$-equivalent.
 \end{proposition}

\begin{proof}  
 Then $\imS:=\overline{\imS_X}=\overline{\imS_Y}$ is an irreducible curve, and let $S(\kappa_1,\kappa_2)$ be its defining polynomial.
   If $\pd{S}{\kappa_2}$ were identically zero, then $K_1|_X$ would  be constant and Lemma~\ref{Lem:ZeroDimSig}  would imply that $\imS$ is a single point. Therefore $\pd{S}{\kappa_2}|_s$ is nonzero for all but finitely many points $s\in \imS$. Moreover, since $X$ and $Y$ are non-exceptional, for all but finitely many such points $s\in \imS$,  none of the points in the preimage $\sigma_X^{-1}(s)$ are exceptional in $X$ and none of the points in the preimage $\sigma_Y^{-1}(s)$ are exceptional in $Y$.  By Chevalley's Theorem (see e.g. \cite[Thm. 3.16]{Harris}), the images $\imS_X$ and $\imS_Y$ are constructible sets and thus all but at most finitely many points of their Zariski closure $\imS$. We fix a point $s\in\imS$ with these desired properties, a point $p\in \sigma_X^{-1}(s)$ and a point $q\in \sigma_Y^{-1}(s)$. Otherwise $\overline{\imS_X}$ (and, therefore, $\overline{\imS_Y}$) is a single point, and we let $p$ and $q$ be any non-exceptional points on $X$ and $Y$, respectively.
   %%%%

In both cases, $\sigma_X(p)=\sigma_Y(q)$, meaning that 
$$K_1(j^r_X(p))=K_1(j^r_Y(q)) \ \text{ and } \ K_2(j^r_X(p))=K_2(j^r_Y(q)).$$
 Since $\I$ is separating  and $p$ and $q$ are non-exceptional,  there exists a group element
$g\in G$ for which $j^r_X(p)$ equals $g\cdot j^r_Y(q)$.

Consider a curve $Z=\overline{g\cdot Y}$.
By Lemma~\ref{lem-non-except}, $Z$ is non-exceptional.
Condition $\overline{\imS_X}=\overline{\imS_Z}$ holds due to Proposition~\ref{Prop:GroupEquiv1}.
Due to the fundamental property of prolongation \eqref{pr-g-comm},  we have $j^r_Z(p) =g\cdot j^r_X(q)$.  Therefore,  $p=g\cdot q\in Z$ and  $j^r_Z(p) =j^r_X(p)$.    We  verified that $X$ and $Z$ satisfy all conditions of Lemma~\ref{ode-lemma}. Then $X=Z=\overline{g\cdot Y}$.
\end{proof}
%%%
Combining Lemma~\ref{Lem:ZeroDimSig}  with  Propositions~\ref{Prop:GroupEquiv1} and~\ref{Prop:GroupEquiv2} we get the following corollary.%%%%
\begin{corol}  If   $X$ and $Y$ have a finite symmetry group, then 
$X$ and $Y$ are $G$-equivalent if and only if their signature polynomials $S_X, S_Y$ are equal up to a   non-zero constant factor.
\end{corol}

%%%%
We are finally ready to state  the first main result of the paper about the existence of a pair of classifying invariants    characterizing  the equivalence classes of generic irreducible algebraic curves:

\begin{theorem}\label{Thm:ClassMain}  Let $r$-dimensional  group $G\subset \PGL(3)$ act on $\mathbb{C}^2$. Then there exists a  pair of  differential invariants $\mathcal{I}=\{K_1,K_2\}$ of  differential order at most $r$, such that  for all integers $d$, where    $ \binom{d+2}{2} -2\geq  r$, there exists a Zariski open subset $\mathcal P_d\subset \C[x,y]_{\leq d}$ such that any curves  $X,Y$  whose defining polynomials lie in $\mathcal P_d$ satisfy:  
\beq\label{eq-th-class-main}X\underset G{\cong}Y\quad\Longleftrightarrow\quad  \overline{\imS_X}=\overline{\imS_Y},\eeq
where $\imS_X$ and $\imS_Y$ are signatures of $X$ and $Y$ based on invariants $\I$, as given by Definition~\ref{def:sig}.
\end{theorem}
%%%
\begin{proof} From Theorem~\ref{th-class}  we know that there exists a classifying set $\I$ of rational differential invariants of order at most $r$. By   Propositions~\ref{Prop:GroupEquiv1} and~\ref{Prop:GroupEquiv2}, the statement \eqref{eq-th-class-main} is valid for all $\I$-non-exceptional curves.  By  Theorem~\ref{thm-non-excep-generic}, for any $d$, such that  $ \binom{d+2}{2} -2\geq r$, there exists a Zariski open subset $\mathcal P_d\subset \C[x,y]_{\leq d}$, such that all curves whose defining polynomials  lie in $\mathcal P_d$ are non-exceptional.
\end{proof}
The next theorem  establishes an important relationship between the size of the symmetry group of an algebraic curve  and some properties of  its signature map. This result plays a crucial role in our degree formula derived in the  next sections.  
%%%
\begin{theorem}\label{Thm:NToOne}
Suppose that $\Chi$ is a non-exceptional curve with respect to a classifying set of rational differential invariants $\mathcal{I}=\{K_1,K_2\}$ for action $G$. Then $|\Sym(X,G)|=n$ if and only if the map $\s_{\Chi}$ is generically $n:1$.
\end{theorem}

\begin{proof} 
($\Rightarrow$)    We need to show that there exists a dense subset $\imS_0\subset \overline{{\imS}_{X}}$, such that $|\s_{\Chi}^{-1}(s)|=n$ for all $s\in \imS_0$.   Denote $H:=\Sym(\Chi,G)$. Since $H$ is finite,  from Lemma~\ref{Lem:ZeroDimSig}, it follows  that  $\overline{{\imS}_{X}}$ is an irreducible curve and its defining polynomial $S(\kappa_1,\kappa_2)$ depends non-trivially on $\kappa_2$.  Therefore the set $\imS_1=\left\{s\in \imS_X\,\Big |\, \left.\frac{\partial S}{\partial \kappa_2}\right |_s\neq 0\right\}$ is dense in $\overline{\imS_{\Chi}}$.
Due to Proposition~\ref{pr:FOrbit}  for all but maybe finitely many points $p\in \Chi$,  the orbit  $Hp$ consists of exactly $n$ distinct points. Moreover, since $X$ has only  finitely many exceptional points, the set of points $$X_0=\{p\in X\,|\,Hp \text{ consists of exactly $n$ non-exceptional points}\}$$ is dense in $\Chi$.  Then its image  $\imS_2=\s_{\Chi}(X_0)$ is dense in $\overline{\imS_{\Chi}}$.   It follows that the intersection $\imS_0:=\imS_1\cap\imS_2$ is dense in $\overline{\imS_{\Chi}}$. For any $s\in \imS_0$,  let $p\in \s_{\Chi}^{-1}(s)$. By Lemma~\ref{self-equiv},  $\s_{\Chi}^{-1}(s) =Hp$ and so $|\s_{\Chi}^{-1}(s)|=n$. 

$(\Leftarrow)$ Suppose that the map $\s_{\Chi}$ is generically $n:1$. 
Then, by Lemma~\ref{Lem:ZeroDimSig}, $\Sym(X,G)$ is finite. By the forward implication,  $n=\Sym(X,G)$.
\end{proof}

%%%

\begin{example}
Consider the special Euclidean group $\SE(2)$  of  complex translations and rotations of $\C^2$.
The set  $\ISE=\{K_1,K_2\}$, where $K_1=\kappa^2$, the square of Euclidean curvature, and $K_2=\kappa_s$ its derivative with respect to 
Euclidean arc-length, explicitly given in \eqref{eq-kappa} is classifying.  
Indeed, one can check directly that $\ISE$ separates orbits on the  $\SE$-invariant open subset 
$$
W_2=\left\{\left(x,y,y^{(1)},y^{(2)},y^{(3)}\right)\, |\, \left(y^{(1)}\right)^2+1\neq 0\right\}
$$
and $K_1$ separates orbits on an open set $W_1=\pi(W_2)\subset J^2$ under the standard projection $\pi\colon J^3\to  J^2$. Thus the conditions of Definition~\eqref{def-dsep} are satisfied. According to Theorem~\ref{th-class} we conclude that 
$$
\C(J^3)^{\SE(2)}=\C(K_1,K_2) \ \ \text{ and }  \ \  \C(J^2)^{\SE(2)}=\C(K_1).
$$

By Theorem~\ref{thm-non-excep-generic}, a generic curve of degree $\geq 2$ is non-exceptional with respect to $\ISE$. In fact, 
a careful consideration of the conditions in  Definition~\ref{Def:IReg} shows that there are \emph{no} irreducible curves of degree greater than one that are  $\ISE$-exceptional. 

We will now compute the signature polynomial for  the ellipse $X$ defined by the zero set of $$F(x,y)=x^2+y^2+xy-1.$$ 
The signature map
  $\s_{\Chi}=\left(K_1|_{\Chi},K_2|_{\Chi}\right)\colon X\to \C^2$ is explicitly defined by
\[
  K_1|_{\Chi}(x,y)=36\frac{(x^2+xy+y^2)^2}{(5x^2+8xy+5y^2)^3} \ \ \text{ and } \ \ 
  K_2|_{\Chi}(x,y)= 54\frac{(y^4-x^4+xy^3-x^3y)}{(5x^2+xy+y^2)^3}.
\]
  
Under  the $\SE(2)$-action the ellipse has  a  symmetry group of cardinality two generated by the $180^\circ$-degree rotation.
We observe that  in agreement with Theorem~\ref{Thm:NToOne}, $\s_{\Chi}$ is generically a $2\colon1$ map on $X$.
One can use a Gr\"obner basis elimination algorithm to compute a signature polynomial of $X$, that is  an irreducible polynomial  vanishing on the image of rational map $\s_{\Chi}$:
$$S_{\Chi}(\kappa_1,\kappa_2)=2916\kappa_1^6+972\kappa_1^4\kappa_2^2+108\kappa_1^2\kappa_2^4+4\kappa_2^6-13608\kappa_1^5+1944\kappa_1^3\kappa_2^2+2187\kappa_1^4.
$$
Any curve $\SE(2)$-equivalent to $\Chi$ will have the same signature polynomial.  For most degree three algebraic curves, it  takes much longer  to compute their signature polynomials under $\SE(2)$ actions, and for  higher degree curves it is rarely possible in practice. For this reason, it is of interest to determine properties, such as the degree, of signature polynomials for curves without their explicit computation.
\end{example}

\section{The degree of the signature of  algebraic curves}\label{Sec:Algebraic}

This section is devoted  to the degree formula for the signature polynomial.
In Section~\ref{Sec:AlgGeoBack} we give the necessary algebraic geometry background. 
In Section~\ref{Sec:SigDegForm} we give a formula for the degree of the signature polynomial for a non-exceptional curve with finite symmetry group. (Theorem~\ref{Thm:Main1B})   and some easily computable bounds for  this degree (Corollary~\ref{Cor:Bounds}).
%%%%%
\subsection{Multiplicity, plane curves, and rational maps}\label{Sec:AlgGeoBack}
Here we review and establish some fundamental properties of plane curves, their intersections, and their images under rational maps.  See, for example, \cite{Fulton} or \cite{Shaf1} for more background.

\begin{definition}\label{def-mult}
Given a point $p\in \C^2$, the \emph{local ring} of $\C^2$ at $p$, denoted $\mathcal{O}_p$, is the  
ring of rational functions in $\C(x,y)$ that are defined at $p$. 
A polynomial ideal $I\subset \C[x, y]$ defines an ideal $I\cdot\mathcal{O}_p$ of the local ring, 
and  the \emph{multiplicity of $I$ at $p$} is defined to be the dimension (as a $\C$-vector space) of the quotient:
$$
m_p(I) \ \ = \ \ \dim_\C \left(\mathcal{O}_p/I\cdot\mathcal{O}_p\right) 
$$
In particular, $m_p(I)$ is positive if and only if $p$ belongs to the variety $V(I)$. 
For a homogeneous ideal $J\subset \C[x_0,x_1,x_2]$ and a point $\pp =[p_0:p_1:p_2]\in\C\PP^2$ with $p_i\neq 0$, 
we define the multiplicity of $J$ at $\pp$, denoted $m_{\pp}(J)$, to be $m_p(I)$, where $p\in \C^2$ and $I$ are obtained from $\pp$ and $J$ by restricting 
$p_i$ and $x_i$ to equal $1$, respectively. On can check that this definition is independent of the choice of non-zero coordinate $p_i$. 
\end{definition}

 In an important special case  when the ideal $I$ is generated by two polynomials, $I = \langle F,G \rangle$, 
 we call $m_p(I) = m_p(F,G)$ the \emph{intersection multiplicity} of $F,G$ at $p$. 
In this case, $m_p(F, G) =1$ if and only if $p\in V(F, G)$ and $\nabla_p F$ and $\nabla_p G$ are linearly independent. 

An equivalent definition of multiplicity uses power series. 
After a change of coordinates, we can take $p=(0,0)$. Then $m_p(I)$ equals the dimension as a $\C$-vector space
of the quotient of the power series ring $\C[[x,y]]$ by the image of $I$ in this ring: 
\[
m_p(I)  \ \ = \ \
\dim_\C \
\left(\C[[x,y]] / I\cdot \C[[x,y]] \right).
\]

  To precisely compute intersection multiplicities at a non-singular point $p$,  one can parametrize a neighborhood of $p$ in $V(F)$ 
using \emph{Laurent series}, $\C((t))$. The ring of Laurent series consists of  formal sums
$\sum_{j=k}^{\infty} a_j t^j$  for some integer $k\in \Z$. 
The   series we consider will converge for $t\in \C^*$ of sufficiently small modulus. 
We define the \emph{valuation}, denoted $\val(\mathbf{a})$, of a Laurent series $\mathbf{a} = \sum_{j=k}^{\infty} a_j t^j$
to be the smallest power of $t$ with nonzero coefficient.  Since $p$ is non-singular,  $\pd{F}{x}(p)\neq 0$ or $\pd{F}{y}(p)\neq 0$. Assume $\pd{F}{y}(p)\neq 0$, then in a neighborhood of $p$, $V(F)$ can be  parametrized by $\alpha(t)=(p_1+t, a(t))$, otherwise  it can be parametrized by  $\alpha(t)=( a(t), p_2+t)$, where in both cases $a(t)$ is some Maclaurin series.  Then according to \cite[\S8.4]{Fischer}:
\beq\label{eq-val-mult}
m_p(F, G) =\val(G(\alpha)).
\eeq

We now establish some basic facts and notation about rational maps on $\C\PP^2$. 
A vector $\bphi=[\bphi_0,\bphi_1,\bphi_2]$, whose entries $\bphi_0,\bphi_1,\bphi_2\in \C[x_0, x_1, x_2]$ are homogeneous polynomials of the same degree $d$, 
defines a rational map  $\bphi:\C\PP^2\dashrightarrow\C\PP^2$ (denoted  by the same symbol).  For any non-zero homogeneous polynomial $h \in \C[x_0, x_1, x_2]$, a polynomial vector $h\bphi=[h\bphi_0,h\bphi_1, h\bphi_2]$ defines
an equivalent rational map,  i.e.~the values of the maps $\phi$ and $h\phi$ are equal, whenever both are defined.
In what follows we  \emph{do not} assume that $\gcd(\bphi_0,\bphi_1,\bphi_2)=1$ and the following definition clearly depends on the choice of a polynomial vector.
%%%
\begin{definition} A vector $\bphi=[\bphi_0,\bphi_1,\bphi_2]$ whose entries $\bphi_0,\bphi_1,\bphi_2\in \C[x_0, x_1, x_2]$ are homogeneous polynomials of the same degree $d$,  is called a \emph{homogeneous vector of degree $d$} and  the notation $\deg(\bphi)=d$ is used. 
The \emph{base locus} of $\bphi$  is the set of points at which all its components are zero 
$$
Bl(\bphi)=V(\bphi_0,\bphi_1,\bphi_2).
$$
We say that $\bphi$ is \emph{defined on an algebraic curve} $X$ if $X$ is not contained in $Bl(\bphi)$. 
We say that $\bphi$ is  \emph{non-constant on an algebraic curve} $X$, if the corresponding rational map  $\bphi:\C\PP^2\dashrightarrow\C\PP^2$ is non-constant when restricted to $X$.
\end{definition}
%%%%
%%%
\begin{proposition} \label{Prop:BertiniCor}
Let $\F\in \C[x_0,x_1,x_2]$ be irreducible and homogeneous, 
and let $\bphi=[\bphi_0,\bphi_1,\bphi_2]$ be a homogeneous vector that is both
defined and non-constant on $V(\F)$. 
For $\ba=[a_0,a_1,a_2]\in \C\PP^2$, consider an equivalence class (up to scaling by a constant) of a linear form $\La = a_0y_0 + a_1y_1 + a_2 y_2\in \C[y_0,y_1,y_2]$ and its pullback $\bphi^*\La=a_0\bphi_0 + a_1\bphi_1 + a_2\bphi_2 \in \C[x_0,x_1,x_2]$.
Then for all $\ba\in \C\PP^2$:
\begin{itemize}
\item[(a)] $V(\bphi^*\La) = \bphi^{-1}(V(\La)) \cup Bl(\bphi)$.
\end{itemize}
In addition, if $\ba\in \C\PP^2$ is generic:
\begin{itemize}
\item[(b)] $\F$ and $\bphi^*\La$ have no common factors,
\item[(c)] if $\pp \in V(\F)\cap V(\bphi^*\La)$ with $\pp\not\in Bl(\bphi)$, then $m_{\pp}(\F, \bphi^*\La)=1$. 
\end{itemize}

\end{proposition}

\begin{proof}
%\begin{itemize}

%\item[(a)] 
(a) If $\pp \not\in Bl(\bphi)$, then $\bphi$ is defined at $\pp$. Then $\bphi(\pp)$ belongs to $V(\La)$ if and only if $\pp$ belongs to $V(\bphi^*\La)$. 
If $\pp$ belongs to $Bl(\bphi)$, then it clearly also belongs to $V(\bphi^*\La)$.

%\item[(b)]
(b) Since $\F$ is irreducible, $\bphi^*\La$ and $\F$ have a common factor if and only if $\F$ divides $\bphi^*\La$, if and only if  $\bphi^*{\bf L_\ba}$ is identically zero on $V(\F)$.  Consider a map $\Psi\colon\C\PP^2\times V(\F)\to\C$, defined by $\Psi(\ba,\pp)=a_0\bphi_0 (\pp)+ a_1\bphi_1(\pp)+ a_2\bphi_2(\pp)$.  Since $\bphi$ is defined on $V(\F)$, there exists $\hat\ba\in \C\PP^2$ such that  $\hat a_0\bphi_0 + \hat a_1\bphi_1 + \hat a_2\bphi_2$ is not identically zero on $V(\F)$ (otherwise with  an appropriate choice of $\ba$'s we can show that  $\phi_j\equiv0$ on $V(\F)$ for  $j=0,1,2$).   Then $\Psi(\hat\ba,\hat\pp)\neq 0$ for some $\hat\pp\in V(\F)$.
 By continuity,  $\Psi(\ba,\pp)\neq 0$ in  some open neighborhood  of $(\hat\ba,\hat \pp)$ in $\Psi\colon\C\PP^2$. Thus $\bphi^*\La$ is  non-zero l on $V(\F)$ for all $\ba$ in some open subset of $\C\PP^2$.

%%%%%
%\item[(c)]
(c) It is sufficient to show that $ \nabla\F|_{\pp}$ and $ \nabla_x \bphi^*\La|_{(\pp,\ba)}$ are linearly dependent,
 for a generic $\ba\in \C^3$ and $\pp\in V(\F)$ such that  $\pp\not\in Bl(\bphi)$, where,  for a moment, we  consider $V(\F)$,  $Bl(\bphi)$  to be varieties in $\C^3$.  
Consider
\[\mathcal{Y}   \ = \  
\left\{ (\pp,\ba)\in \C^6\,|\, \p\in V(\F,\bphi^*\La),\quad \p\not\in Bl(\bphi) \cup V({\bf F})_{sing}\right\} \]
where $ V({\bf F})_{sing}$ denotes the variety $V(\frac{\partial \F}{\partial x_0},\frac{\partial \F}{\partial x_1},\frac{\partial \F}{\partial x_2})$.  
Let $\pi:\mathcal{Y} \rightarrow \C^3$ denote the regular map defined by projection $\pi(\p,\ba) = \ba$. 
Note that restricting to $\p\not\in Bl(\bphi) \cup V({\bf F})_{sing}$ makes $\mathcal{Y}$ nonsingular. 

Bertini's generic smoothness theorem  \cite[Ch. 2, Sec. 6, Thm 2.27]{Shaf1}\footnote{ Bertini's generic smoothness theorem is an algebraic analogue of  Sard's  theorem in differential geometry.} then guarantees 
the existences of a nonempty Zariski-open set $U\subset\C^3$ 
so that for all $\ba\in U$ and all preimages $(\p,\ba)\in \pi^{-1}(\ba)$, 
the induced map on tangent spaces $d_{(\pp,\ba)}\pi:T_{\mathcal{Y},(\pp,\ba)}\rightarrow T_{\C^3,\ba}$ is surjective. 

For $(\pp,\ba)\in \mathcal Y$,
 \[T_{\mathcal{Y},(\pp,\ba)} =\ker
\left[
\begin{array}{cccc}
  \nabla\F|_{\pp}& 0& 0& 0     \\
  \nabla_x \bphi^*\La|_{(\pp,\ba)}&  \phi_0(\pp) &\phi_1(\pp)&\phi_2(\pp) 
\end{array}
\right].
\]

 The map  $d_{(\pp,\ba)}$ maps  $(u_0,u_1,u_2,w_0,w_1,w_2)^T\in T_{\mathcal{Y},(\pp,\ba)} $ to $(w_0,w_1,w_2)^T\in T_{\C^3,\ba}=\C^3$. If $ \nabla\F|_{\pp}$ and $ \nabla_x \bphi^*\La|_{(\pp,\ba)}$ are linearly dependent as vectors in $\C^3$, then $(u_0,u_1,u_2,w_0,w_1,w_2)^T$ belongs to $T_{\mathcal{Y},(\pp,\ba)} $ if and only if $(u_0,u_1,u_3)^T\in \ker   \nabla\F|_{\pp}$ and $a_0w_0+a_1w_1+a_2w_2=0$. The latter gives a  non-trivial linear condition on the vectors in the image of $d_{(\pp,\ba)}\pi$ and, therefore,
 if $ \nabla\F|_{\pp}$ and $ \nabla_x \bphi^*\La|_{(\pp,\ba)}$ are linearly dependent,    $d_{(\pp,\ba)}\pi$ is not surjective.  

Combining  the results of the previous two  paragraphs, we conclude that for $\ba\in U$ and $\p\in V(\F)\cap V(\bphi^*\La) $, such that $ \p\not\in Bl(\bphi) \cup V({\bf F})_{sing}$,
 $ \nabla\F|_{\pp}$ and $ \nabla_x \bphi^*\La|_{(\pp,\ba)}$ are linearly independent. 
  Observing  that  for a generic $\ba$, $V(\bphi^*\La)\cap V({\bf F})_{sing}=\emptyset$, we finish the proof.
%\end{itemize}
\end{proof}
%%%%

%%%%%
\begin{lemma} \label{Lem:GenMult}
Let ${\bf F} \in \C[x_0,x_1,x_2]$ be homogeneous and irreducible, and let 
$\bphi$ be a homogeneous vector. For $\pp\in V(\F)\cap Bl(\bphi)$, 
the minimum of $m_{\pp}(\F,a_0\bphi_0 + a_1 \bphi_1 + a_2\bphi_2)$ over all $\ba\in \C\PP^2$ 
is achieved generically. 
\end{lemma}
%%%%%
\begin{proof}
 If $\pp$ is a non-singular point of $V(\F)$, 
then for any $j\in \Z_{\geq 0}$ the collection of $\G$ for which $m_{\pp}(\F, \G)\geq j$ 
is linear subspace of $\C[x_0,x_1,x_2]_D$, where $D=\deg(\G)$ (This claim  easily follows from \eqref{eq-val-mult}. See also \cite[Prob. 3.20]{Fulton}). It follows that 
$m_{\pp}(\F, a_0\bphi_0 + a_1 \bphi_1 + a_2\bphi_2)\geq j$ is a linear condition on $\ba\in \C\PP^2$. 

 Now suppose $\pp$ is a singular point of ${\bf X} = V(\F)$ and consider a non-singular model ${\bf Y}$ of this curve 
with birational morphism $f: {\bf Y}\rightarrow {\bf X}$ (see \cite[Ch. 7]{Fulton}).  
This induces an embedding of the  fields of rational functions
$f^*\colon \C({\bf X}) \hookrightarrow \C({\bf Y})$. 
Choose some linear form $\ell\in \C[x_0,x_1,x_2]_1$ 
with $\ell(\pp)\neq 0$.  
 Let $ \textbf{G}' = \textbf{G}/\ell^{\deg(\textbf{G})}$ in $\C({\bf X})$. 
Using \cite[Ch. 7, Prop. 2]{Fulton}:
 $$m_{\pp}(\F, \textbf{G})=\sum_{{\bf q}\in f^{-1}(\pp)} m_{\bf q}({\bf Y}, \textbf{G}'),$$ where $m_{\bf q}({\bf Y},  \textbf{G}')$ is the order of vanishing of $\textbf{G}'$ at the smooth point ${\bf q}\in {\bf Y}$.

This reduces to the non-singular case.  
\end{proof}

The minimum multiplicity in Lemma~\ref{Lem:GenMult} will reappear frequently and we denote it by 
\begin{equation}\label{eq:mapMult}
\mult_{\pp}(\F,\bphi) \ \ = \ \ \min_{\ba\in \C\PP^2}m_{\pp}(\F, a_0\bphi_0 + a_1 \bphi_1 + a_2\bphi_2).
\end{equation}
The following bounds can be useful for computing this multiplicity:
\begin{proposition}\label{prop-bounds}
Let ${\bf F} \in \C[x_0,x_1,x_2]$ be an irreducible homogeneous polynomial and $\bphi$ be a homogeneous vector defined on $V(\F)$. 
For $\pp\in Bl(\bphi)$ and for any $\ba = [a_0:a_1:a_2]\in \C\PP^2$,
\[ m_{\pp}(\langle \F, \bphi_0, \bphi_1, \bphi_2 \rangle) \ \leq  \  \mult_{\pp}(\F,\bphi) \ \leq \ m_{\pp}(\F, a_0\bphi_0 + a_1\bphi_1 + a_2\bphi_2), \]
where the right inequality is tight for generic $\ba\in \C\PP^2$. 
\end{proposition}
\begin{proof}
For the first inequality, note that for any $\ba\in \C\PP^2$, $\bphi^*\La=a_0\bphi_0 +a_1\bphi_1 + a_2\bphi_2$ 
 belongs to the ideal $\langle \bphi_0, \bphi_1, \bphi_2\rangle$. 
 It follows immediately from Definition~\ref{def-mult}  that  for 
any pair of nested homogeneous ideals $I \subset J \subset \C[x_0,x_1, x_2]$ and any point $\pp\in \C\PP^2$, we have that $m_{\pp}(I) \geq m_{\pp}(J)$.
Therefore, for every point $\pp\in \C\PP^2$,  $m_{\pp}(\F, \bphi^*\L) \geq m_{\pp}(\langle \F, \bphi_0, \bphi_1, \bphi_2\rangle)$. 
The inequality then follows from a generic choice of  $\ba\in\C\PP^2$ and equation \eqref{Eq:AGdegrees}.

The second inequality follows directly from the definition of $\mult_{\pp}(\F,\bphi)$, and tightness follows from Lemma~\ref{Lem:GenMult}.
\end{proof}

\begin{theorem}\label{Thm:MainLem}
Let $\F\in\C[x_0, x_1, x_2]$ be irreducible and homogeneous and $\bphi$ be a homogeneous vector, such that
 the rational map  $\bphi:\C\PP^2 \dashrightarrow\C\PP^2$ is defined and generically $n:1$ on $V(\F)$.
Let $\textbf{P}\in\C[y_0,y_1,y_2]$ denote the minimal polynomial vanishing on the image  $\bphi(V(\F))$.
Then 
\begin{equation}\label{Eq:AGdegrees}
n\cdot \deg(\textbf{P}) 
 \ \ =\ \ 
  \deg(\F)\cdot \deg(\bphi)-\sum\limits_{\pp\in Bl(\bphi)} \mult_{\pp}(\F,\bphi) 
\end{equation}
\end{theorem}

\begin{proof}
For  a linear form  ${\La} = a_0y_0 + a_1y_1 + a_2 y_2\in \C[y_0,y_1,y_2]$, by
Bezout's Theorem (\cite[\S 5.3]{Fulton}) and Proposition~\ref{Prop:BertiniCor}(a) give that 
$$
\deg(\F)\cdot \deg(\bphi^*\La)
 =  \sum_{\pp}m_{\pp}(\F,\bphi^*\La)
 =  \sum_{\pp\in \bphi^{-1}(V(\La))}m_{\pp}(\F,\bphi^*\La)+\sum_{\pp\in Bl(\bphi)}m_{\pp}(\F,\bphi^*\La).
$$
By  Lemma~\ref{Lem:GenMult}, $m_{\pp}(\F,\bphi^*\La)=\mult_{\pp}(\F,\bphi)$  for a generic $\ba\in \C\PP^2$ and  every $\pp$. 
Also, for a generic $\ba$, $\bphi^*\La$ is nonzero and its degree equals $\deg(\bphi)$. 
By Proposition~\ref{Prop:BertiniCor}(c), for each point $\pp\in V(\F)\cap \bphi^{-1}(V(\La))$, the intersection multiplicity $m_{\pp}(\F,\bphi^*\La)$ equals one. 
Since $\bphi$ is generically $n:1$, there are at most finitely many points $\pp\in V(\F)$ for which $|\bphi^{-1}(\bphi(\pp)) \cap V(\F)| \neq n$, implying that  foe a generic $\ba$, the  line $V(\La)$ 
does not contain the image $\bphi(\pp)$ of any of these points. 
Therefore, for every point $\pp\in \bphi^{-1}(\La)\cap V(\F)$,  
there are exactly $n$ points of $V(\F)$ in the set $\bphi^{-1}(\bphi(\pp))$. 
Putting this all together gives that 
$$
\sum_{\pp\in  \bphi^{-1}(V(\La))}m_{\pp}(\F,\bphi^*\La)  
\ \ = \ \  |V(\F)\cap \bphi^{-1}(V(\La))|
\ \ = \ \ n\cdot |\bphi(V(\F))\cap V(\La)|.
$$
By Chevalley's Theorem (see e.g. \cite[Thm. 3.16]{Harris}), the image $\bphi(V(\F))$ is all but finitely many points of its Zariski closure $V(\textbf{P})$. 
For a generic  $\ba$,  every point in $V(\La)\cap V(\textbf{P})$ belongs to $V(\La)\cap \bphi(V(\F))$
and that the number of these points equals to $\deg(\textbf{P})$. This proves equality in \eqref{Eq:AGdegrees}. 
\end{proof}

\subsection{The degree of the signature polynomial}\label{Sec:SigDegForm}

\begin{definition}
Let $X\subset \C^2$ be an algebraic plane curve and let $\psi: X\dashrightarrow \C^2$ be a rational map. 
We say that a rational map $\bphi:\C\PP^2\dashrightarrow \C\PP^2$ is a \emph{projective extension} of $\psi$ 
if 
\[\psi(p) \ \  = \ \ \left(\frac{\bphi_1(1,p)}{\bphi_0(1,p)}, \frac{\bphi_2(1,p)}{\bphi_0(1,p)}\right)\]
for a Zariski-dense set of points $p\in X$ at which $\psi$ is defined and $\bphi_0(1,p)\neq 0$. 
\end{definition}

Recall from Section~\ref{Sec:Sig}, that a classifying set of rational differential invariants of the action of a group $G$ on $\C^2$ 
define a \emph{signature map} $\sigma_X$ on a non-exceptional, irreducible curve $X\subset \C^2$.  As in Definition~\ref{def:sig}, 
we fix a classifying set of rational differential invariants $\I$ with respect to the action $G$ and suppose that the 
signature map $\s_{\Chi}:\Chi\dashrightarrow \C^2$ is non-constant on $X$. 
We will consider a projective extension $\bs:\C\PP^2 \dashrightarrow \C\PP^2$. Note that while we will drop $X$
from the notation, the map $\bs$ still heavily depends on the original curve $X$.

\newpage
\begin{theorem}\label{Thm:Main1B}
Let $X\subset \C^2$ be a non-exceptional algebraic curve defined by an irreducible polynomial $F$, and let  
$n=|\Sym(X,G)|$.  Then for any homogeneous vector $\bs$, defining a projective extension $\bs:\C\PP^2 \dashrightarrow \C\PP^2$ of the signature map $\sigma_X$, 
 the degree of the signature polynomial $S_{\Chi}$ satisfies 
\begin{equation}\label{Eq:MainThmEq}
n\cdot \deg(S_{\Chi}) \ \ = \ \ \deg(\F)\cdot \deg(\bs)-\sum_{{\pp}\in Bl(\bs)} \mult_{\pp}(\F, \bs).
\end{equation}
Here $\F\in \C[x_0,x_1, x_2]$ denotes the homogenization of $F$. 
\end{theorem}

\begin{proof}
From Theorem~\ref{Thm:NToOne} we know that $\s_{\Chi}:\Chi\dashrightarrow \C^2$ is generically $n:1$ map. 
Then $\bs$ is defined and generically $n:1$ on $V(\F)$, which is the Zariski-closure of $X$ in $\C\PP^2$.  
Since $F$, and thus $\F$, are irreducible, the minimal polynomial ${\bf P}$ 
vanishing on the image $\bs(V(\F))$ is also irreducible. Its dehomogenezation is exactly the signature polynomial $S_X$. 
The result then follows from Theorem~\ref{Thm:MainLem}.
\end{proof}

At first glance the last term in the degree formula \eqref{Eq:MainThmEq}  appears to be difficult  to obtain as we recall from \eqref{eq:mapMult},
$\mult_{\pp}(\F,\bphi)$ is defined as the minimal multiplicity over  $\ba\in \C\PP^2$.  The following corollary shows that a generic choice of $\ba\in \C\PP^2$ gives the desired minimal multiplicity, and thus the degree of the signature can be computed by randomized  algorithms.  It also  establishes the degree bounds, that can also help in determining the degree of a signature curve.

\begin{corol}\label{Cor:Bounds}
Under the hypotheses of Theorem~\ref{Thm:Main1B}, for any $\ba\in \C\PP^2$,  we have 
\beq\label{upper-bound}
n\cdot \deg(S_{\Chi}) \  \geq \  \deg(\F)\cdot \deg(\bs)-\sum_{\pp\in Bl(\bs)} m_{\pp}(\F, a_0\bs_0 + a_1\bs_1 + a_2\bs_2),
\eeq
with equality holding for a generic $\ba$. In addition:
\beq\label{lower-bound} n\cdot \deg(S_{\Chi}) \  \leq \  \deg(\F)\cdot \deg(\bs)-\sum_{\pp\in Bl(\bs)} m_{\pp}(\F, \bs_0, \bs_1, \bs_2),  \eeq

\end{corol}
\begin{proof}  This is a direct corollary of Proposition~\ref{prop-bounds}  and Theorem~\ref{Thm:Main1B}.\end{proof}

In the following example we show how one can use the bounds in Corollary~\ref{Cor:Bounds} to predict the degree of the signature polynomial and what problems can arise.

\begin{example}
We will illustrate Theorem~\ref{Thm:Main1B}  and Corollary~\ref{Cor:Bounds} by studying 
the signature of the curve $\Chi$ defined by the zero set of the irreducible cubic $$F(x,y)=x^2y+y^2+y+\frac{64}{121}$$
for  the action of the affine  group $\Af(2)$ consisting of linear transformations and translations on $\C^2$. 
We will use classifying invariants \eqref{eq:InvariantsJet}  introduced in Section~\ref{subsec:Groups} below. If we restrict these invariants to $\Chi$ and cancel common factors, then we can construct a projective extension $\bs$ of $\sigma_X$ where $\deg(\bs)=26$.

In Figure~\ref{Fig:GenLineAndPullback} in red, on the left, the real affine points of $\Chi$ are shown,  while on  the right, the real affine points of its signature curve $\overline{\imS_X}$. In blue, on the right, is the line $V(\La)$ defined by $\ba=[5:1:1]$ and on the left  its pullback  $V(\bs^*\La)$. Under the action of the affine group of transformations on the plane, $\Chi$ has a symmetry group of size two. Then by Theorem~\ref{Thm:NToOne}, the map $\bs$ is generically $2\colon 1$ on $X$.

A direct computation of  the rightmost terms in \eqref{upper-bound} and \eqref{lower-bound} give that 
$$ \sum_{\pp\in Bl(\bs)} m_{\pp}(\F, 5\bs_0 + \bs_1 + \bs_2)=\sum_{\pp\in Bl(\bs)} m_{\pp}(\F, \bs_0, \bs_1, \bs_2)=30$$
This allows us to conclude that
$\sum_{{\pp}\in Bl(\bs)} \mult_{\pp}(\F, \bs)=30.$
Thus by Theorem~\ref{Thm:Main1B} the degree of the signature curve equals 
$\deg(S_{\Chi}) = (3\cdot 26-30)/2 = 24.$

\begin{figure}[H]
  \centering
    \includegraphics[scale=.4]{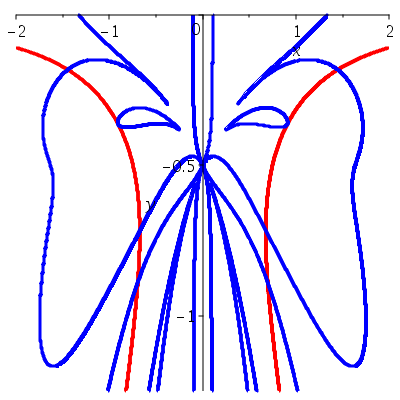}\hspace{.8in} \includegraphics[scale=.4]{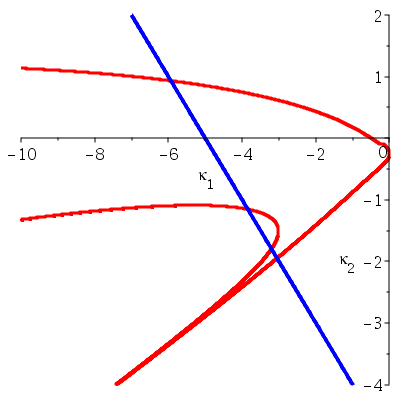}
  \caption{${\color{red}\Chi}$ and ${\color{red}\overline{\imS_X}}$ intersected with ${\color{blue} V(\bs^*{\La})}$ and ${\color{blue}V(\La)}$ respectively.}
  \label{Fig:GenLineAndPullback}
\end{figure}

We now show that  a  line $\Lat$ defined by $\tilde \ba=[1:-6:1]$ does not provide us with exact degree count (the corresponding pictures are given by Figure~\ref{Fig:GenLineAndPullback2}). For this choice of line, 
$ \sum_{\pp\in Bl(\bs)} m_{\pp}(\F, \bs_0-6 \bs_1 + \bs_2)=32$
and  Corollary~\ref{Cor:Bounds} tells us only that $23\leq \deg(S_{\Chi})\leq 24$
and that $\tilde \ba$ is non-generic. Indeed, $V(\Lat)$ intersects $\overline{\imS_X}$ at the point $[0:6:1]$ which is not in $\mathcal{S}_{\Chi}$, a property that must be avoided by generic lines. 
\begin{figure}[H]
  \centering
    \includegraphics[scale=.4]{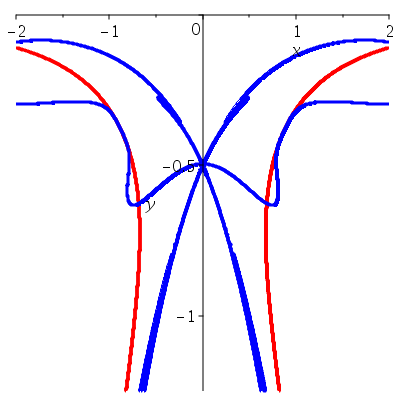}\hspace{.8in} \includegraphics[scale=.4]{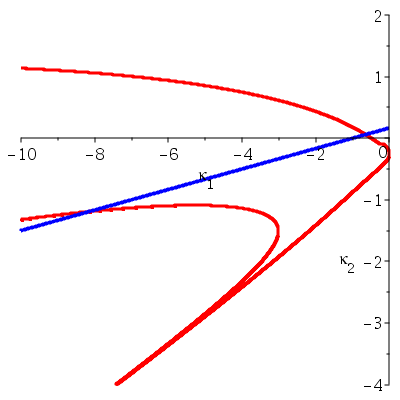}
  \caption{${\color{red}\Chi}$ and ${\color{red}\overline{\imS_X}}$ intersected with with ${\color{blue} V(\bs^*{\Lat})}$ and ${\color{blue}V(\Lat)}$ respectively.}
  \label{Fig:GenLineAndPullback2}
\end{figure}
\end{example}
%%%%%%
\subsection{Super signature and the generic degree}\label{Sec:super}
%%%%%
Let $\displaystyle{F_{\bc}=\sum_{0\leq i,j \leq d} c_{ij}x^iy^j}$ be a polynomial of degree $\leq d$ with unspecialized  coefficients $c_{ij}\in \C$, where $0\leq i , j \leq d$ and $\bc=(c_{00}:c_{10}:\hdots : c_{0d})$.  It is natural to ask if we could  compute  a {signature polynomial} $S_\bc(\kappa_1,\kappa_2)$ for a curve defined by a polynomial with  unspecified coefficients $\bc$ and what information it encodes. In theory, such \emph{super-signature} polynomial can be  defined in the same way as  signature polynomials for specific curves were defined  in Section~\ref{Sec:Sig} and  computed by elimination. 
In practice, the explicit computation seemed to only be feasible  for  small  groups and small $d$, such as, for instance, quadratics under the special Euclidean action. We also know that specialization does not always commute with elimination and, therefore, we can not
 expect that substitution of a specific value  $ \bc=\bc_0$ into the super-signature polynomial will produce a signature of an algebraic curve $X_{\bc_0}$ defined by  $F_{\bc_0}$ even if $X_{\bc_0}$ happens to be an irreducible  non-exceptional curve. However, we can show that this is the case  generically.

 %%%%
\noindent To give a rigorous definition of the super-signature polynomial, we view 
$$F(\bc,x,y)=\sum_{0\leq i,j \leq d} c_{ij}x^iy^j$$ as  a polynomial of degree $d+1$ in $\C[\bc,x,y]$, while $F_{\bc}\in \C[x,y]_{\leq d}$ denotes its specialization. Then $\mathcal Y=V(F)$ is a variety in  $\C^N\times \C^2$, where $N=\left(\begin{matrix} d+2\\ 2 \end{matrix}\right)$.
\beq\label{eq-Y}\mathcal Y=\{(\bc,x,y)\,|\, \bc\in \C^N, (x,y)\in V(F_\bc)\}. \eeq
Let
$j^n(F):\mathcal Y \dashrightarrow\C^{n+2}$ be the rational map defined by the rational functions of the partials of $F_\bc(x,y)$ as in \eqref{Eq:PartialDeriv}, with $\bc$ treated as parameters. For a differential function $K$, let 
$$K|_F=K\circ j^n :\mathcal Y \dashrightarrow\C.$$
For a classifying pair of invariants $\I=\{K_1,K_2\}$, consider the rational map $\sigma:\mathcal Y\dashrightarrow \C^N\times \C^2$ defined by
\beq\label{eq-sig}\sigma(\bc,x,y)=\left(\bc,K_1|_P(\lambda,x,y),K_2(\bc,x,y)\right).\eeq
 Denote the minimal polynomial vanishing on the image $\sigma(\mathcal Y)$ as 
$S(\bc,\kappa_1,\kappa_2)\in\C[\bc,\kappa_1,\kappa_2]$ and let
$\mathfrak{S}=V(S)\subset   \C^N\times \C^2$  be its variety.  We call $S$ a \emph{super-signature polynomial} and $\mathfrak{S}$ the \emph{super-signature variety}. 

The following theorem asserts  that for a generic curve $X=V(F)$ of degree fixed $d$, one can substitute the coefficients of $F$ into the super signature polynomial $S$ to obtain $S_X$ of $X$. 
 %%%%%
\begin{theorem}\label{Thm:Specialization} Let $S(\bc,\kappa_1,\kappa_2)$  be  the \emph{super-signature polynomial} for polynomials of a degree $d$, sufficiently large so that non-exceptional curves are generic\footnote{Theorem~\ref{thm-non-excep-generic} guarantees  that for a sufficiently large $d$ a generic curve is non-exceptional.},  under the action of  a group $G$ with a chosen set of classifying invariants $\mathcal I=\{K_1,K_2\}$. For  ${\bc}\in  \C^N$, let $X_{\bc}=V(F_\bc)$ be the corresponding algebraic curve in $\C^2$. The set of points 
$$\left\{\bc\in\C^N| S(\bc,\kappa_1,\kappa_2)=S_{X_{\bc}}(\kappa_1,\kappa_2)\right\}$$ 
is Zariski dense in $\C^N$, where  $S_{X_{\bc}}$ is a signature polynomial of  the curve  $X_{\bc}$\footnote{Recall that for an irreducible curve the signature polynomial is uniquely defined up to  multiplication by a non-zero constant.}.
\end{theorem}

%
%%%%%%%
\begin{proof}
The variety $\mathcal Y$ defined by \eqref{eq-Y} is irreducible and so is its image $\sigma(\mathcal Y)$ under the rational map \eqref{eq-sig}. Thus the super-signature polynomial $S(\bc,\kappa_1,\kappa_2)\in  \C[x,y,\bc]$, which is a minimal polynomial vanishing on $\sigma(\mathcal Y)$ is irreducible. By Chevalley's Theorem, the image $\sigma(\mathcal Y)$ is an open dense subset of the super-signature variety $\ssig=V(S)$. Therefore, since $\dim\ssig=N+1$,  there exists a variety $\mathcal{H}\subset \ssig$ such that $\dim\mathcal{H}\leq N$  and $\sigma(\mathcal Y)\supset \ssig\backslash\mathcal{H}$.

Consider a regular  map $\pi:\ssig\rightarrow \C^N$ given by $\pi(\bc,\kappa_1,\kappa_2)=\bc$. From the definition of $\ssig$, it is clear that $\pi$ is surjective. We claim that, for a generic $\bc\in \C^N$, the set  $\pi^{-1}(\bc)\cap \mathcal{H}$ is either empty or finite. Indeed, if   $\overline{\pi(\mathcal{H})}\neq \C^N$, then for a generic $\bc$ lying in  the Zariski open non-empty subset  $\C^N\backslash\overline{\pi(\mathcal{H})}$, the set $\pi^{-1}(\bc)\cap \mathcal{H}$ is empty. If  $\overline{\pi(\mathcal{H})}=\C^N$, then, for a generic choice of $\bc$, the dimension of $\pi^{-1}(\bc)\cap \mathcal{H}$ is given by $\dim\mathcal{H}-N=0$ \cite[Ch. 1, Sec. 5, Theorem 1.25]{Shaf1}, implying that  $\pi^{-1}(\bc)\cap\mathcal{H}$ is either empty or finite. 

 By our assumption on $d$, for a generic point $\bc\in \C^N$, the curve $X_{\bc}$ is irreducible  and  non-exceptional (reducible curves have codimension $d-1$).   Let us fix such generic $\bc_0$ that also satisfies the generic condition in the previous paragraph.  
 As before, let $\sigma_{X_{\bc_0}}\colon X_{\bc}\dashrightarrow\C^2$  denote the signature map  of $X_{\bc_0}$, as given in Definition~\ref{def:sig}. Let $\mathcal Y|_{\bc=\bc_0}$ be the slice of the variety $\mathcal Y$, $\ssig|_{\bc=\bc_0}$  the slice of the  super signature variety and  $S|_{\bc=\bc_0}(\kappa_1,\kappa_2)=S(\bc_0,\kappa_1,\kappa_2)$ be the specialization of the super-signature polynomial.  Then $\ssig|_{\bc=\bc_0}=\{(\bc_0,\kappa_1,\kappa_2)\,|\,(\kappa_1,\kappa_2)\subset V(S|_{\bc={\bc_0}})\}$. Let 
$$\mathcal Z_0=\{(\bc_0,\kappa_1,\kappa_2)\,|\,(\kappa_1,\kappa_2)\subset V(S_{X_{\bc_0}}(\kappa_1,\kappa_2))\}.$$ Then
  $$\sigma(\mathcal Y|_{\bc=\bc_0})\subset \mathcal Z_0\subset  \ssig|_{\bc=\bc_0} \text{ and }  \ssig|_{\bc=\bc_0} \backslash\mathcal  Z_0\subset \left(\pi^{-1}(\bc_0) \cap\mathcal H\right).$$ 
 Since  $\pi^{-1}(\bc_0) \cap\mathcal H$ is at most finite by our assumption on $\bc_0$, it follows that $\mathcal Z_0$ is dense in  $\ssig|_{\bc=\bc_0} $  and, therefore,  $V(S_{X_{\bc_0}}(\kappa_1,\kappa_2))= V(S|_{\bc={\bc_0}})$. 
 \end{proof}�
An immediate corollary of the above theorem is that the signature polynomials of generic curves of fixed degree $d$ share the same monomial support in $\kappa_1,\kappa_2$, and hence have  the same degree. Since signature polynomials (up to overall scaling) characterize equivalence classes of generic curves  of degree $d$, it follows that if we consider the super-signature polynomial as an element of $\C(\bc)[\kappa_1,\kappa_2]$ and divide it by one of its non-zero coefficients $h(\bc)\in \C[\bc]$, the coefficients of the resulting polynomial generate the ring of rational invariants for the action of $G$ on the space  of polynomials $\C[x,y]_{\leq d}$.

Since explicit computation of  such generating sets is known to be a very challenging problem, it is not surprising that    computing super-signature polynomials is also very challenging. Conics under $\SE(2)$ is one of the few examples where the super-signature polynomial can be computed explicitely.

%%%%
\begin{example} Consider an arbitrary quadratic curve

$$
F_{\bc}=c_{00}+c_{10}x+c_{01}y+c_{20}x^2+c_{11}xy+c_{02}y^2.
$$
Let  $\Upsilon_1= c_{02}+c_{20}$,  $\Upsilon_2=4c_{20}c_{02}-c_{11}^2$, and $\Upsilon_3= 4\,c_{{00}}c_{{01}}c_{{20}}-c_{{00}}{c_{{11}}}^{2}-
{c_{{01}}}^{2}c_{{20}}+c_{{01}}c_{{10}}c_{{11}}-c_{{01}}{c_{{10
}}}^{2}$. These are known polynomial invariants  for conics under the $\SE(2)$-action.
For the action of the special Euclidean group $\SE(2)$, using the classifying pair of invariants \eqref{eq:InvariantsJet}  introduced in Section~\ref{subsec:Groups}, the super-signature for conics computed by an elimination algorithm is:

\begin{align*}
S(\bc,\kappa_1,\kappa_2)=2916 &\left( \Upsilon_3\right) ^{2}\kappa_1^6+2916\left(\Upsilon_3\Upsilon_1 \left( 4\left(\Upsilon_1\right)^2 -3\Upsilon_2\right)\right)\kappa_1^5+972 \left( \Upsilon_3 \right) ^{2}\kappa_1^4\kappa_2^2\\
&+729\left( \Upsilon_2 \right) ^{3}\kappa_1^4-972\left(\Upsilon_3\Upsilon_2\Upsilon_1 \right)\kappa_1^3\kappa_2^2+108 \left( \Upsilon_3 \right)^{2}\kappa_1^2\kappa_2^4+4 \left( \Upsilon_3 \right) ^{2}\kappa_2^6.
\end{align*}

Dividing through by $(\Upsilon_3)^2$ produces three distinct  non-constant coefficients listed below with constant multiples omitted: 
\beq \nonumber A_1=\frac{\Upsilon_1 \left( 4\left(\Upsilon_1\right)^2 -3\Upsilon_2\right)}{\Upsilon_3}, \, \quad A_2=\frac{\left( \Upsilon_2 \right) ^{3}}{(\Upsilon_3)^2},\quad A_3=\frac{ \Upsilon_1\,\Upsilon_2 }{\Upsilon_3}.
\eeq 
 This is a generating set for the  field of rational invariants for the action of $\SE(2)$ on the space  of quadratic polynomials, but it is not a minimal generating set  because $A_1=4\frac{(A_3)^3}{A_2}-3\,A_3$.

\end{example}
%%%%

Although computing a super-signature is very challenging, we can use super-signatures to establish theoretical results. Below we  use Theorem \ref{Thm:Specialization} to show that the generic degree  is  the sharp upper bound   for the degrees of signature polynomial.
  Discussion and further implications of the above theorem are explored in \cite{MikeThesis}.
 %%%%
 \begin{theorem}
 \label{Thm:Upperb} Under the assumptions of Theorem~\ref{Thm:Specialization}, for a generic curve of  degree $d$, the degree of its signature polynomial equals to the $(\kappa_1,\kappa_2)$-degree $D$ of the super-signature   polynomial. Moreover, for any  non-exceptional curve $X$ of degree  less or equal than  $d$, the degree of its signature polynomial is less or equal than $D$. 
 \end{theorem}

 %%%
 \begin{proof}
The set of values of $\bc_0\in\C^N$, such that  $\deg S|_{\bc=\bc_0}=D$ is Zariski dense, and so its intersection with the set of  $\bc_0\in\C^N$ in the proof Theorem~\ref{Thm:Specialization}  for which  $ S|_{\bc=\bc_0}$ is the signature polynomial for the curve $X_{\bc_0}$ is also Zariski dense. Thus for a generic curve the degree of its signature polynomial equals to $D$.

To show that $D$ is an upper bound, let $X$ be a non-exceptional curve with defining polynomial $F$, which might not satisfy the generic conditions of the previous paragraph,  and 
  let $Y$ be a non-exceptional curve with defining polynomial $H$, whose coefficients do satisfy these generic conditions. Let  $\bc_0$ and $\bc_1\in \C^N$ denote the vectors of coefficients of polynomials $F$ and $H$, respectively. By our assumptions,  the signature polynomial $S_Y(\kappa_1,\kappa_2)=S(\bc_1,\kappa_1,\kappa_2)$ and $\deg S_Y=D$, but  these conditions may not hold for  $X$ and $\bc_0$.
  
    Consider a pencil of polynomials      $P_\lambda(x,y)=(1-\lambda)F(x,y)+\lambda H(x,y)$ parametrized by  $\lambda \in \C$, let  $\bc_\lambda\in \C^N$ be the corresponding coefficients vector and $X_\lambda\subset\C_2$ be the corresponding algebraic curve. By Theorem~\ref{Thm:Specialization},  there is a Zariski dense neighborhood of $\bc_1$ satisfying the generic conditions of the first paragraph of the proof, and therefore,  $S_{X_\lambda}(\kappa_1,\kappa_2)=S(\bc_\lambda,\kappa_1,\kappa_2)$  and  the degree of  $ S|_{\bc=\bc_\lambda}(\kappa_1,\kappa_2)=S(\bc_\lambda,\kappa_1,\kappa_2)$  is $D$ for   all but finitely many values of $\lambda$, where     $S_{X_\lambda}$ is the signature polynomial of $X_\lambda$.

Since $F$ and $H$ are irreducible, one can easily show that $P(\lambda,x,y)=(1-\lambda)F(x,y)+\lambda H(x,y)$ is irreducible as a polynomial in $\C[\lambda, x,y]$. Let    $\mathcal Z=V(P)\subset \C\times \C^2$ be the irreducible variety it defines. It is easy to verify that 
$$\mathcal Z=\{(\lambda,x,y)\,|\, \lambda\in \C, (x,y)\in V(P_\lambda)\}. $$
Let $j^n(P):\mathcal Z\dashrightarrow\C^{n+2}$ be the rational map defined by the rational functions of the partials of $P_\lambda(x,y)$ as in \eqref{Eq:PartialDeriv}, with $\lambda$ treated as a parameter. For a differential function $K$,  define  $K|_P=K\circ j^n\colon\mathcal Z\dashrightarrow\C$ 

Similarly to the way we introduced  super-signatures in the paragraph preceding  Theorem~\ref{Thm:Specialization}, for a classifying pair of invariants $\I=\{K_1,K_2\}$,  we define a rational map $\tau:\mathcal Z\dashrightarrow \C\times \C^2$ by $$\tau(\lambda,x,y)=\left(\lambda,K_1|_P(\lambda,x,y),K_2|_P(\lambda,x,y)\right).$$ Denote the minimal polynomial vanishing on the image $\tau(\mathcal Z)$ as $Q(\lambda,\kappa_1,\kappa_2)\in\C[\lambda,\kappa_1,\kappa_2]$. Since $\tau(\mathcal Z)$ is the image of an irreducible variety under a rational map, $Q(\lambda,\kappa_1,\kappa_2)$ is irreducible. Since $V(Q)$ is an irreducible variety of dimension $2$ and it  is not equal to the $(\kappa_1,\kappa_2)$-plane in $\C\times \C^2$, for all $\lambda_0\in \C$, $Q|_{\lambda=\lambda_0}(\kappa_1,\kappa_2)=Q(\lambda_0, \kappa_1,\kappa_2)$ is a non zero  polynomial in $\C[\kappa_1,\kappa_2]$. 

For all $\lambda_0$, such that $X|_{\lambda_0}$ is non-exceptional, since specialization of coefficients commutes with differentiation and algebraic operation we have $$\sigma_{X_{\lambda_0}}(X_{\lambda_0})=\{(\kappa_1,\kappa_2)|(\lambda_0,\kappa_1,\kappa_2)\in \tau(\mathcal Z|_{\lambda=\lambda_0})\}.$$ The irreducible signature polynomial $S_{X_{\lambda_0}}$  and the specialization $Q|_{\lambda=\lambda_0}$ are both zero on this  set. Hence $S_{X_{\lambda_0}}$ divides the non-zero polynomial $Q|_{\lambda=\lambda_0}$.
In particular, for $\lambda_0=0$, we have  $S_{X}$ divides  $Q|_{\lambda=0}$.

Using the same argument as in the second paragraph of the proof of Theorem~\ref{Thm:Specialization}, one can show that  the set $\{\lambda\in \C \,|\,S_{X_{\lambda}}=Q(\lambda, \kappa_1,\kappa_2)\}$ is dense in $\C$. From the third paragraph of the current proof, we know that the set of $\{\lambda\in \C \,|\,\deg S_{X_{\lambda}}=D\}$ is dense in $\C$. Combining these two facts, we conclude that  the $(\kappa_1,\kappa_2)$-degree of $Q(\lambda,\kappa_1,\kappa_2)$ equals to $D$. Then $\deg(Q|_{\lambda=0})\leq D$, and since 
$S_{X}$ divides  $Q|_{\lambda=0}$,  we  conclude that $\deg S_X\leq D$.
 \end{proof}

%%%%%%%%%
\section{Classical subgroups of the projective groups}\label{Sec:Projective}
%%%%%%%
In this section, we apply our general results  to the actions of the full projective group and its  affine, special affine,  and special Euclidean subgroups. In Section~\ref{subsec:Groups} we explicitly list classifying pairs and exceptional curves for each  of these groups. In Section~\ref{Sec:GenericDegree}, we derive the  degree formula for signatures  of  generic curves under these actions as a function of the degree of the original curve (Theorem~\ref{Thm:DegBound}), observe that this dependence is quadratic and   show that these generic degrees are sharp upper bounds. Finally, in Section~\ref{Sect:Fermat}, we use Fermat curves to illustrate that  non-generic curves, in particular curves with a large symmetry group, may have much lower degree than generic curves. For arbitrary degree curves in this family, we give formulas of their  projective and affine signature polynomials and observe that the  degrees of these  signatures do not depend on  the degrees of the original curves.    
%%%%%
\subsection{Classifying invariants}
\label{subsec:Groups}

Here we introduce rational classifying pairs of  invariants for the actions  of $\PGL(3)$ and some of its of well-known subgroups: the affine group $\Af(2)$, 
the special affine  group $\SAf(2)$,
and the special Euclidean group  $\SE(2)$. For the treatment of the full Euclidean  and  the similarity groups see \cite{MikeThesis}.

As we discussed at the beginning of Section~\ref{Sec:Curves}, $\PGL(3)$ is the group of automorphisms of $\CP$ and is isomorphic to the quotient group $\mathcal{G}\mathcal{L}(3)\backslash\{\lambda I\}$, where  $\mathcal{G}\mathcal{L}(3)$ denotes the group of $3\times 3$ non-singular matrices, $\lambda\in\C$ is non-zero and $I$ is the identity matrix. 
The actions of    $\PGL(3)$ and its subgroups on $\CP$ and $\C^2$ are  given by \eqref{p-act} and  \eqref{a-act}. 
\begin{definition}
The \emph{affine group}, denoted $\Af(2)$, is the subgroup of $\PGL(3)$ that fixes the line of points $[x_0:x_1:x_2]$ with  $x_0=0$. 
\end{definition}
The affine group  is isomorphic  to a subgroup  of $\mathcal{G}\mathcal{L}(3)$ of matrices with the  first row equal to  $[1,0,0]$. 
It is  a group of linear transformations and translations on $\C^2$.
\begin{definition}
The \emph{special affine group}, denoted $\SAf(2)$, is the subgroup of $\Af(2)$ that preserves area under the action  \eqref{a-act}.
\end{definition}
The {special affine group} is isomorphic  to a subgroup of $\mathcal{G}\mathcal{L}(3)$ of matrices with the first row equal to  $[1,0,0]$ and the determinant equal to 1.
\begin{definition}

The \emph{special Euclidean group}, denoted $\SE(2)$, is the subgroup of $\PGL(3)$  isomorphic to the group of matrices
$$
\left[\begin{matrix}
1 & 0 & 0\\
a & c & s \\
b & -s & c
\end{matrix}
\right], \qquad \text{with $c^2 + s^2 =1$}. 
$$

\end{definition}
The real subset of $\SE(2)$ is  the well-known special  Euclidean group of rotations and translations  on $\R^2$. 

In  \cite{BKH}, the authors used classical non-rational differential invariants  to build two lowest order rational invariants for the projective and affine groups and directly proved that they satisfy the Definition~\ref{def-dsep} of classifying invariants over $\R$  (see Theorem 4 in \cite{BKH}).  Using the same line of argument, we can show that these invariants are classifying over $\C$, and  also produce classifying pairs for the actions of the special affine and the special Euclidean groups over $\C$.  
 The following inductive expressions \cite{faugeras94, kogan03} for classical differential invariants are useful for expressing these pairs in a concise manner.
We start with the classical  Euclidean curvature and arc-length:
\begin{gather}
\label{eq-kappa}
\kappa=\frac{y^{(2)}}{(1+[y^{(1)}]^2)^{3/2}}, \quad ds=\sqrt{1+[y^{(1)}]^2} dx
\end{gather}
and express the special affine curvature and  arc-length in terms of them:
\begin{gather*}%\label{affc}
\mu=\frac{3 \kappa (\kappa_{ss}+3 \kappa^3)-5 \kappa_s^2}{9 \kappa^{8/3}},
\qquad
d\alpha=\kappa^{1/3}ds,
\end{gather*}
where $\kappa_s=\frac{d\kappa}{ds}=\frac 1{\sqrt{1+[y^{(1)}]^2}}\frac{d\kappa}{dx}$. In a similar manner, the projective curvature and arc-length are 
\begin{gather*}%\label{pc}
\eta=\frac{6\mu_{\alpha\alpha\alpha}\mu_\alpha-7 \mu_{\alpha\alpha}^2-9\mu_\alpha^2 \mu}{6\mu_\alpha^{8/3}},
\qquad
d\rho=\mu_\alpha^{1/3}d\alpha.
\end{gather*}

\begin{theorem}
The following are pairs of classifying invariants for the actions of $\SE(2)$, $\SAf(2)$, $\Af(2)$, and $\PGL(3)$  on $\C^2$: 
\beq\label{eq:InvariantsJet}\begin{tabular}{|c|c|c|c|c|}
\hline
&&&&\\
Group & $\SE(2)$&$\SAf(2)$ &$\Af(2)$&$\PGL(3)$ \\
\hline

&&&&\\
$K_1$& 
$\kappa^2=\dfrac{(\Theta_2)^2}{(\Theta_1)^3}$&
 $\mu^3=\dfrac{(\Theta_4)^3}{(\Theta_2)^8}$&  
 $\dfrac {\mu_{\alpha}^2}{\mu^3}=\dfrac{(\Theta_5)^2}{(\Theta_4)^3}$ &
  $\eta^3=\dfrac{(\Theta_7)^3}{(\Theta_5)^8}$\\
&&&&\\
$K_2$&
$\kappa_s=\dfrac{\Theta_3 \ }{(\Theta_1)^3}$&
$\mu_{\alpha}=\dfrac{\Theta_5 \ }{(\Theta_2)^4}$&
$\dfrac{\mu_{\alpha\alpha}}{\mu^2}=\dfrac{\Theta_6 \ }{(\Theta_4)^2}$&
$\eta_{\rho}=\dfrac{\Theta_8 \ }{(\Theta_5)^4}$\\
&&&&\\
\hline
\end{tabular}
\eeq
The explicit formulas for $\Theta$'s in terms of jet coordinates are given in Table~\ref{table:thetas}.
\end{theorem}
We use $\ISE$, $\ISA$, $\IA$, and $\IP$ to denote the respective pairs of classifying invariants in \eqref{eq:InvariantsJet}.
\begin{proof}
In \cite[Theorem 4]{BKH}, $\IA$, and $\IP$ are shown to be classifying  in the real case. The proof for the complex case follows similarly and an analogous argument  can be applied to $\ISE$ and $\ISA$. See \cite{suppl} for details.
\end{proof}

\begin{table}
{\small
\begin{align*}
\Theta_1&=u_1^2+1\hspace{1in}  \Theta_2=u_2 \hspace{1in}  \Theta_3=u_3\Theta_1 - 3 u_1 \Theta_2^2 \\
 \Theta_4&=3u_4u_2-5u_3^2 \hspace{2.15in} \Theta_5=9u_5u_2^2-45u_4u_3u_2+40u_3^3\\
\Theta_6&=9u_6u_2^3-63u_5u_3u_2^2-45u_4^2u_2^2+255u_4u_3^2u_2-160u_3^4\\
\Theta_7&=(9 /2)\left[18u_7u_2^4(\Theta_5)-189u_6^2u_2^6+126u_6u_2^4(9u_5u_3u_2+15u_4^2u_2-25u_4u_3^2)\right.\\
&-189u_5^2u_2^4(4u_3^2+15u_2u_4)+210u_5u_3u_2^2(63u_4^2u_2^2-60u_4u_3^2u_2+32u_3^4)\\
&\left.-525u_4u_2(9u_4^3u_2^3+15u_4^2u_3^2u_2^2-60u_4u_3^4u_2+64u_3^6)+ 11200u_3^8\right]\\
\Theta_8&=(243/2)(u_2^4)\left[2u_8u_2(\Theta_5)^2-8u_7(\Theta_5)(9u_6u_2^3-36u_5u_3u_2^2-45u_4^2u_2^2\right.\\
&+120u_4u_3^2u_2-40u_3^4)+504u_6^3u_2^5-504u_6^2u_2^3(9u_5u_3u_2+15u_4^2u_2-25u_4u_3^2)\\
&+28u_6(432u_5^2u_3^2u_2^3+243u_5^2u_4u_2^4-1800u_5u_4u_3^3u_2^2-240u_5u_3^5u_2+540u_5u_4^2u_3u_2^3\\
&+6600u_4^2u_3^4u_2-2000u_4u_3^6-5175u_4^3u_3^2u_2^2+1350u_4^4u_2^3)-2835u_5^4u_2^4\\
&+252u_5^3u_3u_2^2(9u_4u_2-136u_3^2)-35840u_5^2u_3^6-630u_5^2u_4u_2(69u_4^2u_2^2-160u_3^4-153u_4u_3^2u_2)\\
&\left.+2100u_5u_4^2u_3(72u_3^4+63u_4^2u_2^2-193u_4u_3^2u_2)-7875u_4^4(8u_4^2u_2^2-22u_4u_3^2u_2+9u_3^4)\right]
\end{align*}} \vspace{-.3in}
\caption{Differential functions used  in \eqref{eq:InvariantsJet}. Here $u_k$ denotes $y^{(k)}$.} \label{table:thetas}
\end{table}

\begin{proposition}\label{Prop:NonExcep}
The exceptional curves with respect to  $\IP$, $\IA$, and $\ISA$ 
 are lines and conics. The $\ISE$-exceptional curves are lines. In particular, if $X=V(F)$ is a curve exceptional with respect to the classifying invariants in (\ref{eq:InvariantsJet}) then $F$ has degree at most two.
\end{proposition}

\begin{proof}
Propositions 2 and 3 from Section 4.3 in \cite{BKH} show that $\IA$- and $\IP$-exceptional curves are lines and conics and an analogous argument shows that this is the case for $\ISA$-exceptional curves as well. A curve $X=V(F)$ being $\ISE$-exceptional is equivalent to the curve satisfying either $F_y\equiv 0$, $\Theta_1\equiv 0$, or $\Theta_2\equiv 0$, all of which imply $X$ is degree one or two.
\end{proof}

\subsection{The generic signature degree}\label{Sec:GenericDegree}

We derive formulas for the degrees of signatures of generic\footnote{As stated in the introduction, we say that a property holds for a generic curve of degree $d$, if there exists a nonempty Zariski-open subset $\mathcal P_d$ of  $\C[x,y]_{\leq d}$,   such that for  all  $F\in \mathcal P_d$ the property holds for $V(F)$.}
 curves for 
 the four actions discussed in Section~\ref{subsec:Groups} with signature maps based on
 the classifying sets $\ISE, \ISA, \IA, \IP$  given in  \eqref{eq:InvariantsJet}.  
To do so we analyze each term in the   degree formula \eqref{Eq:MainThmEq} of Theorem~\ref{Thm:Main1B}.
We  start by  taking a closer look at the  rational functions defining invariants \eqref{eq:InvariantsJet}.   

%%%%%%%%%%
\begin{lemma}\label{Lem:DeltaPolys}
For a generic polynomial $F\in \C[x,y]$ of degree  $d\geq 3$, 
the restrictions of the differential functions $\Theta_i$ to the curve $V(F)$ 
are equal to rational functions of the form $T_i(x,y)/(F_y)^{d_i}$ with $\deg(T_i)\leq \gdeg_i$ where $\tau_i, d_i$ are given 
as follows: 
\begin{center}
\begin{tabular}{|r|c|c|c|c|c|c|c|c|} \hline
$i$			& $1$		&$2$			&$3$			&$4$			&$5$			&$6$			&$7$			&$8$	\\ \hline
$\gdeg_i $	&$2d-2$		&$3d-4$		&$6d-8$ 		&$8d-12$ 		&$12d-18$	&$16d-24$	&$32 d - 48$	&$ 48 d - 72$\\ \hline
$d_i  $		&$2$			&$3$			&$6$ 		& $8$		&$12$		&$16$		&$32$		&$48$\\ \hline
\end{tabular}
\end{center}
\end{lemma}

\begin{proof}
One can check that each derivative function restricted to $\Chi = V(F)$ can be written 
$$
y^{(n)}|_\Chi=\frac{P_n(x,y)}{\left(F_y\right)^{2n-1}} \ \ \text{ where } P_n\in \Q\left[\frac{\partial^{i+j}{F}}{\partial x^i \partial y^j} \ : \ i+j\leq n   \right]
$$
and $P_n(x,y)$ is a polynomial of degree $(2n-1)d-(3n-2)$.
One can evaluate the formulas for $\Theta_1, \hdots, \Theta_8$ given in Table~\ref{table:thetas}. 
For example, plugging in the rational expressions for $y^{(n)}|_X$ to the differential formula for $\Theta_4$ gives 
$\Theta_4 \ = \left( 3P_4 P_2 - 5 (P_3)^2\right) / F_y^{10}$. See \cite{suppl} for explicit computations. 
The numerator has degree $10d-14$, but it is also divisible by $F_y^2$. This gives an expression $\Theta_4 = T_4(x,y)/(F_y)^{8}$
where $T_i$ has degree less than or equal to $8d-12$. 
The arguments for the other differential functions follow similarly.
\end{proof}
Explicit formulas for the polynomials $T_i$ are quite long. A code to compute them can be found in \cite{suppl}. 
Note that for each of the classifying invariants, the partial derivative function $F_y$ cancels out and leaves each invariant as a rational function of the polynomials $T_1,\hdots , T_8$. 
In the following lemma, we use homogenizations of  $T_1,\hdots , T_8$ to write down  projective extensions $\bs$ of the signature maps for each pair   of invariants  \eqref{eq:InvariantsJet}.

%%%%%%
\begin{lemma}\label{Lem:ProjExtAP}
Fix an irreducible polynomial $F\in \C[x,y]$ of degree $d\geq 3$ and let $X = V(F)$. 
For $G = {\mathcal{SE}}, \SAf, \Af, \mathcal P$, let $\sigma_X^G$ denote the signature map given by the invariants $\I^G$ in \eqref{eq:InvariantsJet}. 
Then 
\begin{align}\label{eq:ProjExtSigAP}
\bs^{\mathcal{SE}}&=[\T_1^{\ 3}: x_0^2\T_2^{\ 2}: x_0^2\T_3], &\bs^{\SAf}&=[\T_2^{\ 8}: x_0^4\T_4^{\ 3}: x_0^2\T_2^{\ 4}\T_5],\nonumber \\
\bs^{\Af}&=[\T_4^{\ 3}: \T_5^{\ 2}: \T_4 \T_6], \ \ \ \ \ \ \  \text{ and }&\bs^{\mathcal{P}}&=[\T_5^{\ 8}:\T_7^{\ 3}:\T_8\T_5^{\ 4}]
\end{align}
are projective extensions of the maps $\sigma_X^{\mathcal{SE}}$, $\sigma_X^{\SAf},$ $\sigma_X^{\Af},$ and $\sigma_X^{\mathcal{P}}$, respectively, where for each $i$,
$\T_i $ equals the homogenization, $ x_0^{\gdeg_i}T_i(\frac{x_1}{x_0},\frac{x_2}{x_0}) \in \C[x_0,x_1,x_2]$, of the polynomial $T_i$ from Lemma~\ref{Lem:DeltaPolys}. Moreover,
\[
\deg(\bs^{\mathcal{SE}})  =  6d-6 , \ \deg(\bs^{\SAf}) =  24d-32, \  \deg(\bs^{\Af}) =  24d-36, \  \text{ and } \ \deg(\bs^{\mathcal{P}})= 96d-144.
\]
\end{lemma}
%%%%%
\begin{proof}
First, we note that by Lemma~\ref{Lem:DeltaPolys}, the coordinates of $\bs^{G}$ are homogeneous of the stated degrees and that by Proposition~\ref{Prop:NonExcep}, $X$ is non-exceptional with respect to each of the classifying sets of invariants in (\ref{eq:InvariantsJet}). 
Moreover, with $G = \Af$, for a point $p\in X$  we see that,
\[ 
\sigma_X^{\Af}(p)\ = \
\left( \frac{\Theta_5(p)^2}{\Theta_4(p)^3},  \frac{\Theta_6(p)}{\Theta_4(p)^2}\right) 
\ = \ 
\left( \frac{T_5(p)^2}{T_4(p)^3},  \frac{T_6(p)}{T_4(p)^2}\right) 
\ = \
\left( \frac{\bs^{\Af}_1(1,p)}{\bs^{\Af}_0(1,p)},  \frac{\bs^{\Af}_2(1,p)}{\bs^{\Af}_0(1,p)}\right) 
.\]
Here the middle equality follows from the fact that the factors of $F_y$ given by the degrees $d_i$ in Lemma~\ref{Lem:DeltaPolys} all cancel out in the above expressions. If $\bs^{\Af}(p)$ is not defined then $\Theta_4(p)=0$, meaning $p$ is not $\I$-regular. Thus $\bs^{\Af}(p)$ is defined at all but finitely many points of $X$.
Analogous arguments show that $\bs^{\mathcal{SE}}$, $\bs^{\mathcal{\SAf}},$ and $\bs^{\mathcal{P}}$ are projective extensions of $\sigma^{\mathcal{SE}}_X$, $\sigma^{\mathcal{\SAf}}_X$, and $\sigma^{\mathcal{P}}_X$. 
\end{proof}

We are now ready to analyze the last term in the degree formula \eqref{Eq:MainThmEq}  where the sum of multiplicities is taken  over  the base locus of a  projective extension $\bs$ of the signature map. We first show that, for our choices of projective extensions, all base locus points belonging to a generic curve are ``at infinity.''
%%%%%§

%%%%%%
\begin{lemma}\label{Lem:NoAffineBaseLocusAP}
For a generic polynomial $F\in \C[x,y]_{\leq d}$,\ { $d\geq 4$},
the base locii of the maps $\bs^{\mathcal{SE}}$, $\bs^{\SAf},$ $\bs^{\Af}$ and $\bs^{\mathcal{P}}$ in \eqref{eq:ProjExtSigAP} contain no points of the form $[1:p]\in \C\PP^2$ where $F(p)=0$. 
\end{lemma}
%%%%%%%
\begin{proof}
We will provide a detailed proof for the affine  group and then show how this argument can be adapted to other groups.
For any point $p\in \C^2$, consider the set 
\[
\mathcal{V}^{\Af}_p \ = \ \left\{F \in \C[x,y]_{\leq d} \,|\,  F(p)=0 \text{ and } [1:p] \text{ belongs to the base locus of }\bs^{\Af}\right\}.
\]
Our goal is to show that the set 
$$\mathcal{V}^{\Af}=  \ \bigcup_{p\in \C^2} \mathcal{V}^{\Af}_p $$
has codimension at least 1 in the linear space of polynomials  $\C[x,y]_{\leq d}$.

For a polynomial $F\in \C[x,y]_{\leq d}$,  a point $[1:p]$ belongs to the base locus of the map $\bs^{\Af}$ if and only if $T_4(p) = T_5(p)=0$. 
Polynomials $T_i$ were introduced in Lemma~\ref{Lem:DeltaPolys},  and they can be expressed as  polynomials function of the partial derivatives of $F$. Therefore, for  $ F(x,y)=\sum_{i+j\leq d}c_{ij}x^iy^j$ with undetermined coefficients and  a fixed  point $p$, expressions 
$T_4(p)$ and $T_5(p)$  can be viewed as polynomials in the coefficients $c_{ij}$. This allows us  to express $\mathcal{V}^{\Af}_p $ as the variety
of  three polynomial expressions $F(p)$, $T_4(p)$, and $T_5(p)$ in the coefficients $c_{ij}$ where $i+j\leq d$.

For $p=(0,0)$, we can use computational algebra techniques to find the codimension of this set.  The condition $F(0,0)=0$ is equivalent to $c_{0,0}=0$.
The highest order partial derivative appearing in the expressions for $T_4$ and $T_5$ is 5. Therefore 
 $T_4(0,0)$ and $T_5(0,0)$ can be written as polynomials of $c_{ij}$ where $i+j\leq 5$.  Moreover, for $d\geq 5$,  these polynomials are independent of $d$. For $d=4$, all monomials involving  $c_{ij}, i+j=5$ will disappear.   For $d\geq 4$, one can check (see \cite{suppl}) that  three polynomials $c_{0,0}$,  $T_4(0,0)$ and $T_5(0,0)$  impose algebraically independent conditions, implying  that 
 $\mathcal{V}^{\Af}_{(0,0)} $ has codimension $3$ in $\C[x,y]_{\leq d}$ (the case $d=4$ has to be checked separately). 

Now we claim that for any point $p\in \C^2$, a polynomial $F$ belongs to $ \mathcal{V}^{\Af}_p$ 
if and only if its image under translation $\overline{F}(x,y) = F(x+p_1,y+p_2) $ belongs to $\mathcal{V}^{\Af}_{(0,0)}$. 
Note that the partial derivatives of $F$ are invariant under translations: 
$\frac{\partial^{i+j} \overline{F}}{\partial x^i y^j}(x,y)=\frac{\partial ^{i+j}F}{\partial x^i y^j}(x+p_1,y+p_2)$ for all $i,j$.
Let $\overline{T}_4,\overline{T_5}$ denote the polynomials obtained from Lemma~\ref{Lem:DeltaPolys} from $\overline{F}$. 
Since these are functions of the partial derivatives of $\overline{F}$, they are also invariant under translations: $\overline{T}_i(x,y)=T_i(x+p_1,y+p_2)$.
Then $F$ belongs to $ \mathcal{V}^{\Af}_p$ if and only if 
$F(p) = \overline{F}(0,0) =0$,  $T_4(p) = \overline{T_4}(0,0) =0$, and $T_5(p) = \overline{T_5}(0,0) =0$,  
which occurs if and only if $\overline{F}\in \mathcal{V}^{\Af}_{(0,0)}$. 
This shows that the set of polynomials not satisfying the condition in the statement of Lemma~\ref{Lem:NoAffineBaseLocusAP}
can be written as 
 \[
\mathcal{V}^{\Af} \ = \ \left\{F\in \C[x,y]_{\leq d} \,|\,F(x+p_1, y+p_2)  \in \mathcal{V}^{\Af}_{(0,0)}\right\}.
\]
Then the dimension of $\mathcal{V}^{\Af}$ is at most $\dim(\mathcal{V}^{\Af}_{(0,0)})+2$. Since $\mathcal{V}^{\Af}_{(0,0)}$ has codimension $3$ in 
the space of polynomials $\C[x,y]_{\leq d}$, this means that $\mathcal{V}^{\Af} $ has codimension $\geq 1$.

A similar argument, based on translation of an affine point $p$ to the origin,  goes through for  other groups,
and the proof of the lemma boils down  to showing
that  
\[
\mathcal{V}^G_{(0,0)} \ = \ \left\{F \in \C[x,y]_{\leq d} \,|\, F(p)=0 \text{ and } [1:p] \text{ belongs to the base locus of }  \bs^{{G}}\right\},
\]
where $\bs^{{G}}$ is the projective extension of the signature map for an appropriate group $G$, has a codimension of at least 3 in $\C[x,y]_{\leq d}$.

In the $\mathcal{SE}(2)$ case and $d\geq 2$, the variety $\mathcal{V}^{\mathcal{SE}}_{(0,0)}$ is  defined by four polynomials  $c_{0,0}$, $T_1(0,0)=(c_{01})^2+(c_{10})^2$, $T_2(0,0)=-2c_{20}(c_{01})^2+2c_{11}c_{10}c_{01}-2(c_{10})^2c_{02}$ and $T_3(0,0)$
in $\C[c_{ij} : i+j\leq 3]$.  Clearly, $T_1(0,0)$, $T_2(0,0)$ and $c_{00}$ impose algebraically independent conditions on $\C[x,y]_{\leq d}$. Thus, by the above argument, $\mathcal{V}^{\mathcal{SE}}_{(0,0)}$ must be of codimension at least  3 in $\C[x,y]_{\leq d}$ for all $d\geq 2$.

 In the ${\SAf}(2)$ case,   the variety $\mathcal{V}^{\mathcal{SA}}_{(0,0)}$ is  defined by three polynomials  $c_{0,0}$, $T_2(0,0)$ and $T_4(0,0)$
in $\C[c_{ij} : i+j\leq 4]$, where for $d\geq 4$, $T_2(0,0)$ and $T_4(0,0)$ are independent of $d$. Algebraic independence of these polynomials  is checked in  \cite{suppl}.  Finally,  for the full projective group $\PGL(3)$,  the variety $\mathcal{V}^{\mathcal{SA}}_{(0,0)}$ is   defined by three polynomials  $c_{0,0}$, $T_5(0,0)$ and $T_7(0,0)$
in $\C[c_{ij} : i+j\leq 7]$, where  $T_5(0,0)$ and $T_7(0,0)$ are independent of $d$  for $d\geq 7$. The algebraic independence of these three polynomials is checked in  \cite{suppl}.   When $d=4,5$, or $6$ some monomials disappear and  so algebraic  independence has to be checked separately. \end{proof}
%%%
As a side remark, we point out that under the $\mathcal{SE}(2)$-action, a generic curves does not have any base locus points (even at infinity) as shown in Lemma~\ref{lem:Multiplicities} below. 
%%%%%

\begin{lemma}\label{lem:LocalParamInf}
Let $\F\in \C[x_0,x_1,x_2]_{d}$ be a generic homogeneous polynomial of degree $d$ satisfying two generic conditions: 
\begin{enumerate}
\item [(i)] $\F(0,0,1)\neq 0$, 
\item[(ii)] the discriminant of the univariate polynomial $F(0,1,x_2)$ is nonzero.
\end{enumerate}
Then neighborhood of any point $\pp=[0:p_1:p_2]$ in $V(\F)$ can be parametrized by $t\mapsto [\alpha(t)]$ where 
\beq\label{eq-alpha}
\alpha(t) \  = \  \left( t,  1,  \sum_{j=0}^{\infty}a_j t^j\right) \ \ \in \ \ \C[[t]]^3.
\eeq
Moreover, for any homogeneous polynomial $\G\in \C[x_0,x_1,x_2]$, the intersection multiplicity of $\F$ and $\G$ at $\p$ is given by 
$\val(\G(\alpha))$.
\end{lemma}

\begin{proof}
 Consider a point $\pp=[0:p_1:p_2]$ in $V(\F)$. From the first assumption it follows that $p_1\neq 0$ and  thus we can take ${p_1} = 1$. 
From the second assumption it follows  the restriction  $H = \F(v,1,w) \in \C[v,w]$ satisfies $H_w(0,p_2) = \frac{\partial \F}{\partial x_2}(\p)\neq 0$.  Therefore,  in some neighborhood of $(0,p_2)$, the curve $V(H)\subset \C^2$ agrees with the graph $w = f(v)$ of an analytic function $f$. 
We obtain $\alpha$ as a power series expansion of this function with $a_j = \frac{f^{(j)}(0)}{j!}$.
For the claim that the intersection multiplicity of $\F$ and $\G$ is given by $\val(\G(\alpha))$, 
see \cite[\S8.4]{Fischer}.
\end{proof} 

%%%%%

\begin{lemma}\label{lem:GenParCoefficients} 
For $d\geq 3$, a generic point $(a_0,\hdots, a_{8}) \in\C^9$ can be extended to the coefficients of the parametrization \eqref{eq-alpha} for some $\F\in \C[x_0,x_1,x_2]_d$  satisfying conditions of   
Lemma~\ref{lem:LocalParamInf}. 
\end{lemma}
%%%%
\begin{proof}
Note that $n=8$ and $d\geq 3$ satisfy  the assumptions of   Lemma~\ref{lem:JetGen}, implying  that for a generic  point $a\in \C^9$,
 there exists an  irreducible algebraic curve $X\subset\C^2$ of degree $d$, such that   $(0,a_0)\in X$  and $j^{(8)}_X(0,a_0) \ = \ (0,\, a_0,1!\,a_1\hdots, 8!\,a_8)$.  Let $F(x,y)\in\C[x,y]$ be an irreducible polynomial of degree $d$ whose variety is  $X$. It is easy to check that for  the homogenization 
$\F(x_0,x_1,x_2)=x_1^dF\left(\frac {x_0} {x_1},\frac {x_2} {x_1}\right)$, the projective curve $V(\F)$  has the desired parametrization  \eqref{eq-alpha} in a neighborhood of $[0:1:a_0]$.
\end{proof}
%%%%%%
%%%%%
\begin{lemma}\label{lem:DeltaVal}
Let $F\in \C[x,y]_{\leq d}$ be a generic polynomial with degree $d\geq 3$ 
and let $\alpha=(\alpha_0,\alpha_1,\alpha_2)$ denote the parametrization given by Lemma~\ref{lem:LocalParamInf} for its homogenization $\F$. 
For sufficiently small $t\in \C^*$, the Laurent series 
$$
\beta(t)  = t^{-1}(\alpha_1(t), \alpha_2(t)) =  \left( t^{-1},  \sum_{j=-1}^{\infty}a_{j+1} t^{j}\right)
$$
parametrizes the curve $V(F)$. The differential functions $\Theta_i$ along this parametrization satisfy: 
\begin{center}
\begin{tabular}{|c|c|c|c|c|c|c|c|c|} \hline
$i$					& $1$	&$2$			&$3$			&$4$			&$5$			& $6$	&$7$			&$8$		\\ \hline
$\val(\Theta_i(\beta)) $		&$0$		&$3$			&$4$			&$8$ 		&$15$		&$19$	&$40$		&$ 60$\\ \hline
\end{tabular}
\end{center}
\end{lemma}

\begin{proof}
First let us calculate the image of $\beta$ in the jet space. 
For $(x(t),y(t)) = (t^{-1},t^j)$ with $j\geq 1$, the derivative of $y$ with respect to $x$ equals $-jt^{j+1}$. Repeated applications of $\frac{\partial }{\partial x}$ 
then yields that $y^{(k)}(x)$ equals $(-1)^k t^{j+k}\prod_{i=0}^{k-1}(i+j)$. Then 
for $(x(t),y(t)) = \beta(t)$ and $k\geq 2$, 
$$
y^{(k)}(x) \ \ = \ \  (-1)^k   \cdot \sum_{j=1}^{\infty}a_{j+1} t^{j+k} \cdot \prod_{i=0}^{k-1}(i+j).
$$

We can then evaluate the differential functions $\Theta_1, \hdots, \Theta_8$ on truncations of these formulas, 
where $a_j$ are indeterminates. (See \cite{suppl}.)
For example, evaluating $\Theta_4$ and $\Theta_5$ give
\begin{align*}
\Theta_4(\beta)    \ &= \ -36\cdot a_2 \cdot t^{8} + \text{higher order terms, \ \ \ and  } \\ 
\Theta_5(\beta)   \ &= \  -4320 \cdot (2 a_3^3 - 3 a_2 a_3 a_4 + a_2^2 a_5) \cdot t^{15} + \text{higher order terms.}  
\end{align*}
In each case, the leading coefficients are polynomials of $a_0, \hdots, a_8$. Therefore, by Lemma~\ref{lem:GenParCoefficients} and the genericity of $F$, 
we may assume that these leading coefficients do not vanish. 
\end{proof}
%%%

\begin{lemma}\label{lem:Multiplicities}
For a generic homogeneous polynomial $\F\in \C[x_0,x_1,x_2]_{d}$ with $d\geq 3$ and a point $\p = [0:p_1:p_2]$ in $V(\F)$, 
we have  
\[
\mult_{\p}(\F, \bs^{\mathcal{SE}}) = 0, \ \ \mult_{\p}(\F, \bs^{\mathcal{\SAf}}) = 16, \ \ \mult_{\p}(\F, \bs^{\Af}) = 12, \ \ \text{ and }  \ \ \mult_{\p}(\F, \bs^{\mathcal{P}}) = 72,
\]
where $\bs^{\mathcal{SE}}$, $\bs^{\SAf}$, $\bs^{\Af}$ and $\bs^{\mathcal{P}}$ are the  polynomial vectors given by Lemma~\ref{Lem:ProjExtAP} for $F = \F(1,x,y)$ and the corresponding multiplicities are defined by \eqref{eq:mapMult}. 
\end{lemma}
%%%%%
\begin{proof}
Let $\alpha\in \C[[t]]^3$ be the local parametrization of $V(\F)$ guaranteed by Lemma~\ref{lem:LocalParamInf}. 
For each index $i=1, \hdots, 8$, let $v_i$ denote the valuation of $\T_i(\alpha)$. 
By the same lemma and the formulas in Lemma~\ref{Lem:ProjExtAP}, the desired multiplicities are  
\begin{align}
\nonumber \mult_{\p}(\F, \bs^{\mathcal{SE}}) &= \min\{3v_1, 2+2v_2, 2+ v_3\},  &\mult_{\p}(\F, \bs^{\SAf})  &= \min\{8v_2, 4+3v_4, 2+ 4v_2+v_5\},  \\
\label{eq-mult-g} \mult_{\p}(\F, \bs^{\Af})  &= \min\{3v_4, 2v_5, v_4+v_6\}, \text{ and} \hspace{-.3cm}&\mult_{\p}(\F, \bs^{\mathcal{P}})   &= \min\{8v_5, 3v_7, 4v_5+v_8\}.
\end{align}
Let $\beta \in \C((t))^{2}$ be the tuple of Laurent series given by Lemma~\ref{lem:DeltaVal}. 
Since $\T_i$ is homogeneous of degree $\tau_i$  and  $\alpha = t\cdot (1,\beta)$, we see that 
\[\T_i(\alpha) = \T_i(t ,t \beta) = t^{\gdeg_i} \T_i(1,\beta) = t^{\gdeg_i} T_i(\beta).\]
By genericity, the coefficient of $x^{d-1}$ in $F_y$ is nonzero, meaning that 
the valuation of $F_y(\beta)$ is $-(d-1)$. 
This and the formulas $T_i = \Theta_i\cdot (F_y)^{d_i}$ from Lemma~\ref{Lem:DeltaPolys} give that 
\begin{align*}
v_i = \val(\T_i(\alpha))  =  \gdeg_i + \val(T_i(\beta))  &=   \gdeg_i + \val(\Theta_i(\beta)) + d_i \val(F_y(\beta))\\ &=  \gdeg_i+   \val(\Theta_i(\beta)) - d_i (d-1).
\end{align*}
Then combining the data from Lemmas~\ref{Lem:DeltaPolys} and~\ref{lem:DeltaVal} gives that 
\[
v_1=0,\ \ v_2 = 2,  \ \ v_3=2, \ \ v_4=4, \ \ 	v_5 = 9, \ \ v_6 = 11, \ \ v_7 = 24, \ \text{ and }\  v_8 = 36.
\]
Substitution  of this value in \eqref{eq-mult-g} finishes the proof.
%Then the minimums above are achieved by $3v_1=0$, $8v_2=16$, $3v_4 = 12$, and $8v_5 = 72$. 
\end{proof}

\begin{theorem}\label{Thm:DegBound}
Fix an irreducible polynomial $F\in \C[x,y]_{\leq d}$ of degree $d\geq 4$ and let $X = V(F)$. 
Let $S_X^{\mathcal{SE}}$, $S_X^{\Af}$, $S_X^{\SAf}$, and $S_X^{\mathcal{P}}$ denote the signature polynomials defined by the invariants in \eqref{eq:InvariantsJet}. Then, when the symmetry group of $X$ is finite,
\[
\deg(S_X^{\mathcal{SE}}) \leq  6d^2-6d, \ \
\deg(S_X^{\SAf}), \deg(S_X^{\Af}) \leq 24d^2 -48d,
\ \ \text{ and } \ \ 
\deg(S_X^{\mathcal{P}}) \leq 96d^2- 216d.
\]
Furthermore, these bounds are tight for generic $F\in \C[x,y]_{\leq d}$. 
\end{theorem}

\begin{proof}
First we show that the bounds above are achieved for generic $F\in \C[x,y]_{\leq d}$. 
By Proposition~\ref{Prop:NonExcep}, the curve $X$ is non-exceptional for $\ISE$, $\IA$, $\ISA$, and $\IP$ and we can apply Theorem~\ref{Thm:Main1B}. 
Since $X$ is a general curve of degree $\geq 4$, its symmetry group is trivial and so  $n=1$ \cite{Shok94}. 
Let $\F\in \C[x_0,x_1, x_2]$ denote the homogenization of $F$ and $\Chib = V(\F)\in \C\PP^2$. 
Then by Theorem~\ref{Thm:Main1B},  for a projective extension of a  signature map $\bs$, 
\[
\deg(S_X)  \ =\  \deg(\F) \cdot \deg(\bs) - \sum_{\p\in Bl(\bs)}\mult_{\p}(\F, \bs).
\]
For a generic $\F$, the variety $V(\F)$ contains exactly $d$ points with $x_0=0$. The multiplicities  at each of these points is 
given by Lemma~\ref{lem:Multiplicities} for every group under consideration. By Lemma~\ref{Lem:NoAffineBaseLocusAP}, these are the only points of $V(\F)$ in the base locii 
of the projective extensions $\bs^{\mathcal{SE}}$, $\bs^{\Af}$, $\bs^{\SAf}$, and $\bs^{\mathcal{P}}$.
All together, this gives
\begin{align*}
&\deg(S_X^{\mathcal{SE}}) \  = \ d \cdot (6d-6) - d \cdot (0) \   = \  6d^2-6d, \\
&\deg(S_X^{\SAf})  \! \ = \ d \cdot (24d-32) - d \cdot (16) \   = \  24d^2-48d, \\
&\deg(S_X^{\Af})  \ \ = \ d \cdot (24d - 36) - d \cdot (12) \   = \  24d^2 -48 d, \text{ and } \\
&\deg(S_X^{\mathcal{P}})  \ \ = \ d \cdot ( 96 d-144) - d \cdot (72) \  = \  96d^2-216d.
\end{align*}

 From Theorem~\ref{Thm:Upperb}, these degrees are upper bounds.
\end{proof}

%%%%%%%%%%
We note that for all groups we consider,  for generic curves, the  degree of  the signature curve has  a quadratic dependence  on the degree of the original curve. The symmetry group of a generic  curve is trivial, but many interesting and important  curves have non-trivial symmetry groups. In accordance  with  the degree formula \eqref{Eq:MainThmEq}, these curves have  lower degree signature. The next subsection is devoted to the  Fermat curves family. For this family, in the case of  the projective and affine action,  the growth of the signature curve degree is completely suppressed by  the   increase in the  symmetry group size. 
\bigskip

%%%%%%%%%%%
\subsection{The Fermat curves} \label{Sect:Fermat}The $d$-th degree Fermat curve, denoted in this section by $\Chi_d$, is the zero set over $\C^2$ of the polynomial  $F_d(x,y)=x^d+y^d+1$, whose homogenization is $\F_d(x_0,x_1,x_2)=x_0^d+x_1^d+x_2^d$.

\begin{theorem}\label{Thm:FermatSym}
The symmetry group of the $d$-th degree Fermat curve  with respect to full projective, affine and special Euclidean groups are:
\begin{itemize}
\item  ${\rm Sym}(\Chi_d,\PGL(3))=S_3\rtimes \left(\Z_d\times \Z_d\right)$ of cardinality $6d^2$,
\item  ${\rm Sym}(\Chi_d,\Af(2))=S_2\rtimes \left(\Z_d\times \Z_d\right)$ of cardinality $2d^2$, and 
\item  ${\rm Sym}(\Chi_d,\SE(2))=
\left\{
\begin{array}{ccc} \Z_1 &  \quad  & \text{ of cardinality 1, when $d$ is odd \ }  \\
\Z_2\times \Z_2 &  \quad  & \text{ of cardinality 4, when $d$ is even.}  
\end{array}
\right.$
\end{itemize}
Here $S_k$  is the permutation group over $k$-elements and  $\Z_k$ is  the cyclic groups of $k$-elements.
\end{theorem}
\begin{proof}  In \cite{Tz95} it has been shown that ${\rm Sym}(\Chi_d,\PGL(3))$ 
 consists of compositions of permutations of the homogeneous coordinates $[x_0:x_1:x_2]$ and transformations scaling the coordinates by $d$-th roots of unity, \textit{i.e.} $[x_0:x_1:x_2]\to [x_0:\omega_1x_1:\omega_2x_2]$, where $\omega_1$ and $\omega_2$ are $d$-th roots of $1$. This shows the first result.  
 Since  ${\rm Sym}(\Chi_d,\Af(2))$ is the subgroup  of ${\rm Sym}(\Chi_d,\PGL(3))$ that fixes the homogenous coordinate $x_0$, in the second result $S_3$ must be replaced with $S_2$. Finally, in the case of the special Euclidean group for odd $d$  there are no non-trivial symmetries, while for even $d$  the symmetry group is generated  by  two independent elements, each of order two, namely  $[x_0:x_1:x_2]\to [x_0: -x_2: x_1]$ and  $[x_0:x_1:x_2]\to [x_0: -x_1: -x_2]$.
 \end{proof}
For the projective and for the affine groups,  the cardinality of the symmetry groups depend quadratically on $d$. 
At the same time Theorem~\ref{Thm:DegBound} shows that the degrees of generic signature curves depend quadratically on $d$.
In fact, these quadratic dependencies cancel, and the degrees of signatures of the Fermat curves for these actions are independent of $d$. 
%%%
\begin{theorem}\label{Thm:Fermat} The signature of  the Fermat curve $V(x^d+y^d+1) \subset \C^2$  has
\begin{itemize}
\item     degree four for all $d\geq 3$ for the $\PGL(3)$-action.  
\item   degree two for $d=3$ and degree three for all $d\geq 4$  for the $\Af(2)$-action.
\end{itemize}
 \end{theorem}
 We remind the reader that  the signatures of lines and conics are undefined under the projective and affine actions. 
The above  result can be proven  by computing all quantities  involved in \eqref{Eq:MainThmEq} (see \cite{suppl} for details) or by explicit computation of signature polynomials. We present here the explicit formulas for signatures polynomials and observe that their  coefficients (but not their degrees) depend on $d$.
  For the projective group the signature polynomial  of the Fermat curve of degree $d>2$ is:
{\footnotesize
\begin{align*}
S_{\Chi_d}^{\mathcal P}&(\kappa_1,\kappa_2)  \ = \ 
49392 (d-2)^4 d^3 (d+1)^4 (2 d-1)^4 \kappa_2^4+602112 (d-2)^4 d^3 (d+1)^4 (2d-1)^4 \kappa_1 \kappa_2^2 \\
&+10584 (d-2)^3 d^2 (d+1)^3 (2 d-1)^3 \left(10 d^2-3 d+3\right) \left(34 d^2-27 d+27\right) \kappa_2^3\\
&+1835008 (d-2)^4 d^3  (d+1)^4 (2 d-1)^4 \kappa_1^2-9289728 (d-2)^3 d^2 (d+1)^3 (2 d-1)^3\left(d^2-d+1\right)^2 \kappa_1 \kappa_2\\
&+61236 (d-2)^2 d (d+1)^2 (2 d-1)^2 \left(d^2-d+1\right) \left(10 d^2-3 d+3\right)^2 \left(16 d^2-9 d+9\right)\kappa_2^2\\
&-23328 (d-2)^2 d (d+1)^2 (2 d-1)^2 \left(11792 d^8-17376 d^7+28152 d^6-24424 d^5+19473 d^4-8940 d^3\right.\\ &\left.+3358 d^2-324 d+81\right) \kappa_1
+118098 (d-2)(d+1) (2 d-1) \left(d^2-d+1\right)^2 \left(10 d^2-3 d+3\right)^4 \kappa_2\\
   &+531441 d \left(d^2-d+1\right)^3 \left(10 d^2-3 d+3\right)^4.
\end{align*}
}
The  signature polynomial of the Fermat curve of degree $d>2$ under the affine action is:
{\footnotesize\begin{align*}
 S^{\Af}_{\Chi_d}&(\kappa_1,\kappa_2) \ = \ (d-3)^2 (d-2) d^2 (d+1) (2 d-1)^3 \kappa_2^3-(d-5)^3 d (2 d-1)^2  \kappa_1^2\\ 
&+ 3 (d-5) (d-2) d (d+1) (2 d-1)^2 (5 d-11) \kappa_1 \kappa_2+6 (d-2)^2 d (d+1)^2 (2 d-1)^2 \left(d^2-4 d+6\right) \kappa_2^2\\
&+2 (d-2)^2 (d+1)^2 (2 d-1) \left(15 d^2-10 d+18\right) \kappa_1+12 (d-2)^3 (d+1)^3 (2 d-1) \left(d^2-2 d+3\right) \kappa_2\\
&+8 (d-2)^4 d (d+1)^4.
\end{align*}}
\noindent For $d=3$, the coefficient of $\kappa_2^3$ vanishes and the degree of the signature polynomial drops to two.
\section{Discussion and future directions}
%%%%
The problem of equivalence and symmetry of algebraic curves under the action of the projective group and its subgroups is intimately related to the problem of the equivalence and symmetries of ternary forms under the action of the general linear group and its subgroups. Such problems and their generalizations were at the heart of classical 19th century invariant theory.   Linear changes of variables induce linear  transformations of the coefficients of  polynomials. The latter  serve as coordinates on the  $\binom{d+2}{2}$-dimensional vector space $\C[x_0,x_1,x_2]_d$.  The classical problem was to find generators of the rings of polynomial invariants  and generators  of the fields of rational invariants under such actions. {Actions on the product space  $\C[x_0,x_1,x_2]_d\times \C^3$ were also considered,  and the invariants with respect to these actions were called covariants in the classical literature}. An overview of the classical methods for constructing invariants and  covariants as well their application to the
classification of polynomials can be found in \cite{GraceYoung03},
\cite{Gur64}, \cite{O99}.  Due to Hilbert's finite basis theorem, the  generating sets  for such actions
%for  actions of linear reductive groups
 are  finite \cite{HilbertEng}, but  their cardinality and the complexity of the invariants grow dramatically with the degree. In fact, the complete set  of the generators remains unknown except for the ternary forms of low degrees.   

Applications of differential invariants to the problems in classical invariant theory was first proposed by Sophus Lie \cite{Lie1983}.
One of the main advantages of using  differential invariants in comparison with classical algebraic invariants and covariants is that the same set of invariants can be used  for all ternary forms independently of their degrees.
Differential signature constructions for homogeneous polynomials in two variables  (binary forms) was first  introduced  by Olver \cite{O99} and applied to their symmetry groups computation in \cite{BO00}. 
 For the case of  ternary forms, a fundamentals set of differential invariants was first computed  in \cite{Kthesis} and it has been shown in \cite{GPolatOlver17} that the differential algebra of invariants can be generated by a single differential invariant and two invariant differential operators. In his thesis, Wears \cite{wears11}, considered differential signatures of polynomials in an arbitrary number of variables. 
In the above literature, one extends the action to the jets of the graphs of homogeneous polynomials $u=\F(x_0,x_1,x_2)$ or in-homogeneous  polynomials $u=F(x,y)$, computes  the set of fundamental invariants of a sufficiently high order,  and uses these invariants to construct signatures. 
{In contrast to the signatures developed in this paper, the signatures of these graphs are surfaces rather than curves.}
%%%%

\emph{Gaining an understanding of the relationship between  signatures surfaces of the defining polynomials, considered in the above literature,  and signatures curves of their zero  sets, considered in this paper,  is an interesting problem for future research.} In particular, signature surfaces  of the graphs of the Fermat polynomials with respect to the projective groups  computed in \cite{Kthesis} can be compared with the signatures of Fermat curves obtained here. 

%%%%%
Proposition~\ref{prop-choice}, provides a simple relationship between pairs of classifying invariants for a given group. \emph{The signatures curves and their degrees depend on a choice of  classifying invariants, but a careful study of this dependence is outside of the scope of the current paper.}

Since explicit computation of signature polynomials is challenging,  it is helpful to identify  their properties that can be computed  a priori. In this paper we derived the degree formula of signature polynomials. One natural step is to \emph{determine their Newton polytope, which gives  a more detailed information about the monomials of the signature polynomial}.

It is immediate that the signatures curves of rational curves are rational. However,  the signatures of non-rational curves may be also rational, as happens for instance in the case of all Fermat curves under the affine and the projective actions.
%Fermat cubic, $x^3+y^3+1=0$
%under the affine action. 
 It is an interesting problem \emph{to identify classes of curves with rational signatures  and, more generally, to understand if we can predict the genus of a signature curve. }

\bibliographystyle{plain}
%\bibliography{signature}

\end{document}